\documentclass[reqno]{amsart}
\usepackage{amsthm}
\usepackage{amsmath}
\usepackage{amssymb}
\usepackage[]{graphicx}
\usepackage[english]{babel}
\usepackage[utf8]{inputenc}
\usepackage{proof}
\usepackage{hyperref}
\usepackage{tikz}

\newtheorem{theorem}{Theorem}[section]
\newtheorem{lemma}[theorem]{Lemma}
\newtheorem{corollary}[theorem]{Corollary}
\newtheorem{proposition}[theorem]{Proposition}
\theoremstyle{definition}
\newtheorem{definition}[theorem]{Definition}
\newtheorem{example}[theorem]{Example}
\theoremstyle{remark}
\newtheorem{remark}[theorem]{Remark}

\DeclareMathOperator{\FV}{FV}
\DeclareMathOperator{\Vars}{Vars}

\newcommand{\numbere}{\mathrm{e}}

\newcommand{\Base}{\mathcal{B}}
\newcommand{\CNF}{\mathcal{C}}
\newcommand{\DNF}{\mathcal{D}}
\newcommand{\ENF}{\mathcal{E}}
\newcommand{\Formula}{\text{Formula}}
\newcommand{\SSigma}[1]{\boldsymbol{\Sigma_{#1}}}
\newcommand{\PPi}[1]{\boldsymbol{\Pi_{#1}}}

\DeclareMathOperator{\QQ}{Q}

\newcommand{\myinfer}[2]{\frac{#2}{#1}}

\newcommand{\plus}{\mathrel{\text{+}}}

\newcommand{\nplus}[2]{{#1} \oplus {#2}}
\newcommand{\nplusone}[2]{{#1} \oplus {#2}}

\newcommand{\ntimes}[2]{{#1} \times {#2}}

\newcommand{\distrib}[2]{{#1}\ltimes {#2}}
\newcommand{\distribn}[2]{{#1}\ltimes {#2}}

\newcommand{\distribex}[2]{{#1}\propto {#2}}

\newcommand{\distribone}[2]{{#1}\rtimes {#2}}
\newcommand{\distribzero}[2]{{#1}\rtimes {#2}}

\newcommand{\explog}[2]{{#1}\Uparrow {#2}}

\newcommand{\qexplogone}[3]{{#1}\uparrow^{#3} {#2}}
\newcommand{\explogone}[2]{{#1}\uparrow {#2}}
\newcommand{\explogzero}[2]{{#1}\uparrow {#2}}
\newcommand{\explogall}[2]{{#1}\Rsh^{#2}}

\newcommand{\enf}[1]{\left\|{#1}\right\|}
\newcommand{\enfpos}[1]{\left|{#1}\right|}

\DeclareMathOperator{\Dest}{dest}
\DeclareMathOperator{\As}{as}
\DeclareMathOperator{\In}{in}
\newcommand{\dest}[4]{\Dest {#1} \As ({#2}.{#3}) \In {#4}}

\newcommand{\wpair}[2]{\langle{#1},{#2}\rangle}

\begin{document}

\author{Taus Brock-Nannestad and Danko Ilik}
\title[An Intuitionistic Formula Hierarchy Based on High-School Identities]{An Intuitionistic Formula Hierarchy\\ Based on High-School Identities}
\begin{abstract}
We revisit the notion of intuitionistic equivalence and formal proof
representations by adopting the view of formulas as exponential
polynomials.
After observing that most of the invertible proof rules of
intuitionistic (minimal) propositional sequent calculi are formula
(i.e.\ sequent) isomorphisms corresponding to the high-school
identities, we show that one can obtain a more compact variant of a
proof system, consisting of non-invertible proof rules only, and where
the invertible proof rules have been replaced by a formula
normalisation procedure.
Moreover, for certain proof systems such as the G4ip sequent calculus
of Vorob'ev, Hudelmaier, and Dyckhoff, it is even possible to see all
of the non-invertible proof rules as strict inequalities between
exponential polynomials; a careful combinatorial treatment is given in
order to establish this fact.
Finally, we extend the exponential polynomial analogy to the
first-order quantifiers, showing that it gives rise to an
intuitionistic hierarchy of formulas, resembling the classical
arithmetical hierarchy, and the first one that classifies formulas
while preserving isomorphism.


\end{abstract}
\dedicatory{To the memory of Kosta Došen}
\maketitle

\section{Introduction}
\label{sec:intro}

Classical logic has a standard semantics independent of the notion of
proof. One could for instance do model theory without ever involving
proof systems. However, the intended meaning of the logical
connectives for intuitionistic logic makes its semantics inherently
proof-theoretic. For example, the intuitionistic \emph{validity} of
$F\vee G$ amounts to either \emph{having a proof} of $F$ or a proof of
$G$.

Equivalence of formulas is perhaps also more subtle
intuitionistically. 
  First, whereas in classical first-order logic any formula can be
  characterised as belonging at an appropriate level of the
  arithmetical hierarchy through an equivalent formula in prenex form,
  in intuitionistic logic, the existence of an equally versatile
  hierarchy appears to be elusive (see Section~\ref{sec:conclusion}).
  Second, the usual notion of equivalence denoting implication in
  both directions is not semantics-, that is, proof-preserving. For
  instance, the equivalence $F \wedge F \leftrightarrow F$ only allows
  preserving proofs between the left-hand side and the right-hand side
  for some special cases of $F$, but, when $F$ is a disjunction with
  both disjuncts provable, there would be four different possible
  proofs of the left-hand side and only two different possible proofs
  of the right-hand side.

\emph{Isomorphism} of formulas seems to be a better notion when our
aim is to preserve semantics across an equivalence. This strong notion
of equivalence of formulas,
\[F\cong G,\] asks not only that $F\leftrightarrow G$, but also that
there exist proof transformations
$\phi : ~\vdash F \longrightarrow ~\vdash G$ and
$\psi : ~\vdash G \longrightarrow ~\vdash F$, such that given any
proof $\mathcal{D}_1$ of $F$ and $\mathcal{D}_2$ of $G$, we have that
\[\psi(\phi(\mathcal{D}_1))\equiv\mathcal{D}_1 \quad\text{ and }\quad
  \phi(\psi(\mathcal{D}_2))\equiv\mathcal{D}_2.\] However, adopting
isomorphism as the standard intuitionistic notion of equivalence makes
us stumble upon another fundamental problem: having a good definition
of \emph{identity of proofs}, ``$\equiv$'', is itself open since the
early days of intuitionistic proof theory (see \cite{Dosen2003} and Section~\ref{sec:conclusion}).

For all of these reasons, the study of formal proof systems is perhaps
more pressing for intuitionistic logic and constructive mathematics,
than it is for classical mathematics. And, as constructive reasoning
plays an important role in proof assistant software, these problems
are also directly relevant to formal specification and verification of
programs, not to mention the foundations of popular mathematical
theories with intuitionistic cores, such as type theory and topos
theory.

In this paper, we contribute to the intuitionistic proof theory
related to the aforementioned problems by means of the
perspective of \emph{intuitionistic formulas seen as exponential
  polynomials}\footnote{As opposed to ordinary polynomials
  where the exponent is only allowed to be a constant, for exponential
  polynomials the exponent can contain variables or indeed another
  exponential polynomial.}. The paper is organised as follows.

In Section~\ref{sec:highschool}, we shall show that the invertible
proof rules of
sequent calculi for intuitionistic (minimal) propositional logic
present simple isomorphisms that arise from high-school identities,
and that hence one can build a \emph{high-school} variant (HS) of a
sequent calculus, which is complete for provability, but does not need
to contain invertible proof rules. Such a calculus HS is thus a way to
obtain a proof system for intuitionistic (minimal) logic that relies
on the non-invertible rules only; this, for example, can facilitate
comparing two proofs for identity.

Furthermore, we show that one can have an intuitionistic sequent
calculus that allows an interpretation of \emph{non}-invertible rules
as \emph{in}-equalities between exponential polynomials. This is shown
on the example of Vorob'ev, Hudelmaier, and Dyckhoff's sequent
calculus G4ip~\cite{Dyckhoff1992,troelstra2000basic}, in Section~\ref{sec:inequality},
opening the possibility to use arithmetical or analytic arguments in
intuitionistic proof theory; we give an application to termination of
proof search.

In Section~\ref{sec:hierarchy}, we shall show how the analogy between
formulas and exponential polynomials can be extended to the
first-order quantifiers, obtaining a normal form for intuitionistic
first-order formulas and a novel intuitionistic ``arithmetical''
hierarchy that preserves formula isomorphism, and hence identity of
proofs. We believe that the proposed hierarchy is a technical device
that could play a rôle similar to that of the arithmetical hierarchy
for classical logic (which exists since the 1920s), as it is both
simple and semantics-preserving.

Finally, Section~\ref{sec:conclusion} summarises the results and
discusses related work.


\section{Proof Rules as Equalities}
\label{sec:highschool}
\allowdisplaybreaks

Identifying formulas of intuitionistic (minimal) propositional logic with
exponential polynomials -- by writing $F\wedge G$ as $FG$, $F\vee G$
as $F+G$, $F\to G$ as $G^F$, and $\top$ as $1$, and treating atomic formulas as variables
-- allows one to generalise the notion of \emph{validity of equations} in
the standard model of positive natural numbers by the notion of
\emph{formula isomorphism}. Namely, if we take $F\cong G$ as defined
in Section~\ref{sec:intro}, the following implication holds (see
\cite{FioreDCB2006} for a proof):
\[
F\cong G ~\Longrightarrow~ \mathbb{N}^+ \vDash F=G.
\]
That is, if $F$ and $G$ are isomorphic formulas, then the
corresponding arithmetical expressions must be equal for the variables
contained therein interpreted over the positive natural numbers.

Equation validity (and consequently formula isomorphism) poses
interesting meta-theoretic problems that go back to Tarski's
High-School Identity Problem (see
\cite{burris04,FioreDCB2006,Ilik2014}). In particular, validity (and
isomorphism) is not finitely axiomatisable: there is \emph{no} finite
set of equality axioms Ax that is sufficent to \emph{derive} every
valid equation (isomorphism), i.e. such that
$ \mathbb{N}^+ \vDash F=G ~\Longrightarrow~ \text{Ax} \vdash F=G.  $

Nevertheless, this meta-theoretic problem does not concern us in this
paper, since for the purpose of analysing intuitionistic equivalence
(the formula hierarchies presented later) and proof systems, we do
not need to have a complete characterisation of isomorphism, but
merely certain isomorphisms provable from the twelve high-school identities
(HSI axioms):
\begin{align}
  F &= F\\
  F+G &= G+F \\
  (F+G)+H &= F+(G+H)\\
  F G &= G F\\
  (F G) H &= F (G H)\\
  F(G+H) &= F G + F H \label{hsi:distr}\\
  F 1 &= F\label{hsi:multunit}\\
  F^1&= F\label{hsi:powerone}\\
  1^F&= 1\label{hsi:onepower}\\
  F^{G+H}&= F^G F^H \label{hsi:exp:1}\\
  (F G)^H&= F^H G^H \label{hsi:exp:2}\\
  (F^G)^H&= F^{G H}. \label{hsi:exp:3}
\end{align}
That is, when we close these axioms under appropriate equality and congruence
rules (see for instance \cite{Ilik2014}), we can talk about formal
derivability of an equation, $\text{HSI} \vdash F=G$, for which we
have:
\begin{equation}
  \label{eq:hsivalid}
  \tag{*}
  \text{HSI} \vdash F= G ~\Longrightarrow~ F\cong G.  
\end{equation}
Every equation derivable from the HSI axioms can thus be also seen as
establishing a strong intuitionistic equivalence.

This correspondence between formulas and exponential polynomials
suggests an investigation of the rules of intuitionistic proof systems
as rules for transforming exponential polynomials. Let us start with
the invertible\footnote{For the purpose of this paper, we adopt a
  stronger notion of invertibility than the one usual in structural
  proof theory. We not only require that the premise is provable from
  the conclusion, but moreover that the premise and the conclusion are
  isomorphic as formulas (sequents). This stronger notion of
  invertible rule is also known as an \textit{asynchronous} rule in
  the terminology of focusing sequent
  calculi~\cite{LiangMiller2007}. However, the impact of the
  terminology change is mild: the only proof rule from the paper which
  is invertible in the classic sense but not in ours is ($\to_l^P$).}
rules of the intuitionistic propositional sequent calculus, written
out in two columns, the left one giving a rule in formula notation,
while the right one gives the same rule in exponential polynomial
notation.
\begin{align}
  \tag{$\to_r$} \myinfer{\Gamma\vdash F\to G}{F,\Gamma\vdash G} &&& \myinfer{\left(G^F\right)^\Gamma}{G^{F\Gamma}}\\[1em]
  \tag{$\wedge_r$} \myinfer{\Gamma\vdash F\wedge G}{\Gamma\vdash F \quad \Gamma\vdash G} &&& \myinfer{(FG)^\Gamma}{F^\Gamma G^\Gamma}\\[1em]
  \tag{$\vee_l$} \myinfer{(F\vee G),\Gamma \vdash H}{F,\Gamma \vdash H \quad G,\Gamma\vdash H} &&& \myinfer{H^{(F+G)\Gamma}}{H^{F\Gamma} H^{G\Gamma}}\\[1em]
  \tag{$\wedge_l$} \myinfer{(F\wedge G), \Gamma\vdash H}{F,G,\Gamma\vdash H} &&& \myinfer{H^{FG\Gamma}}{H^{FG\Gamma}}
\end{align}
While we used the traditional symbols ``$\vdash$'' and comma for the
formula (sequent) notation, one should keep in mind that for us the
``$\vdash$'' is just an implication (implication being
right-associative, we could have written for instance
$\Gamma\vdash F\to G$ as $\Gamma\to F\to G$), while the comma is just
a conjunction.

We can thus see that the usual invertible rules of intuitionistic
sequent calculi correspond to polynomial simplification rules. And,
since these last ones are either instances or compositions of
high-school identities, by (\ref{eq:hsivalid}) the invertible proof
rules are also valid as formula-, that is, sequent isomorphisms. The
same is true for the additional invertible rules sometimes present,
such as the ones of the sequent calculus G4ip~\cite{Dyckhoff1992,troelstra2000basic}:
\begin{align}
  \tag{$\to_l^\wedge$}\myinfer{(G\wedge F \to H),\Gamma\vdash I}{(F\to G\to H),\Gamma\vdash I} &&& \myinfer{I^{H^{GF}\Gamma}}{I^{(H^G)^F\Gamma}}\\[1em]
  \tag{$\to_l^\vee$}\myinfer{(F\vee G\to H),\Gamma\vdash I}{(F\to H),(G\to H),\Gamma\vdash I} &&& \myinfer{I^{H^{F+G}\Gamma}}{I^{H^F H^G \Gamma}}.
\end{align}

In order to express a complete version of the intuitionistic sequent
calculus in terms of exponential polynomials, we need to consider the
non-invertible proof rules as well. In particular, it suffices to
consider the remaining rules of G4ip:
\begin{align}
  \tag{axiom} \myinfer{P,\Gamma\vdash P}{} &&& \myinfer{P^{P\Gamma}}{} \\[1em]
  \tag{$\vee^1_r$} \myinfer{\Gamma\vdash F\vee G}{\Gamma\vdash F} &&& \myinfer{(F+G)^\Gamma}{F^\Gamma} \\[1em]
  \tag{$\vee^2_r$} \myinfer{\Gamma\vdash F\vee G}{\Gamma\vdash G} &&& \myinfer{(F+G)^\Gamma}{G^\Gamma} \\[1em]
  \tag{$\to_l^P$} \myinfer{(P\to F),P,\Gamma\vdash G}{F,P,\Gamma\vdash G} &&& \myinfer{G^{F^PP\Gamma}}{G^{F P \Gamma}} \\[1em]
  \tag{$\to_l^\to$} \myinfer{((F\to G)\to H),\Gamma\vdash I}{(G\to H),\Gamma\vdash F\to G \quad H,\Gamma\vdash I} &&& \myinfer{I^{H^{G^F}\Gamma}}{(G^F)^{H^G \Gamma}I^{H\Gamma}},
\end{align}
where $P$ denotes a prime (i.e. atomic) formula.\footnote{We do not
  give a special treatment for intuitionistic absurdity, or negation;
  $\bot$ would be treated as an atomic proposition as any other, and
  $\neg$ would be replaced by implication into $\bot$. Note that the
  identification of $\bot$ with the polynomial $0$ would be
  problematic, since in the presence of $0$ the simple correspondence
  between isomorphism and equation validity that we use breaks
  (see~\cite{FioreDCB2006}). Note also that even in such a setting
  known as minimal logic~\cite{SchwichtenbergW2011}, one can recover
  the proof rule \textit{ex falso quodlibet} as soon as some form of
  induction axiom is present (see, for instance, Section~7.1.3
  of~\cite{SchwichtenbergW2011}).}

Due to the absence of contraction in G4ip, all of the non-invertible
rules $\myinfer{G}{F}$ satisfy the arithmetic inequality $F \le G$
when variables are interpreted in $\{n\in\mathbb{N}\mid n\ge 2\}$. We
shall prove this in Section~\ref{sec:inequality}, since the case of
$(\to_l^\to)$ is not obvious.

The goal for the present section is to derive from G4ip a proof
system that does not contain any of the (bureaucratic and
non-informative, from the identity of proofs perspective) invertible
rules. This system, written with the help of exponential polynomial
notation is called the \emph{high-school} variant of G4ip
(HS). Proofs in HS only consist of the translations of the
informative rules of G4ip. However, that the procedure for
deriving HS is generic, and could be performed on another version of
the intuitionistic sequent
calculus.

The starting idea is to use the analytic transformation,
\[
G^F = \numbere^{F\log G} = \exp(F\log(G)),
\]
in order to decompose binary exponentiation (i.e.\ implication) in
terms of unary exponentiation and the logarithmic function, just as
the approach to normal forms in exponential fields
\cite{hardy}.\footnote{This hints at interpreting implication
  ``classically'', that is, in terms of two distinct ``negation''
  symbols, $\neg_{\exp}$ and $\neg_{\log}$, such that
  \[
    F\to G := \neg_{\exp}(F \wedge \neg_{\log}G),
  \]
  but we do not pursue this superficial analogy further in this paper.}

As already analysed in a previous study of the
$\beta\eta$-equations for terms of normalised type \cite{explog}, the
exp-log decomposition of implication leads to a normal form of
propositional formulas that is obtained by left-to-right rewriting
using the high-school identities (isomorphisms) (\ref{hsi:exp:1}),
(\ref{hsi:exp:2}), (\ref{hsi:exp:3}), and (\ref{hsi:distr}). 

There is some liberty in determining the order in which to apply the
equations. One precise and structurally recursive procedure for
computing the normal form of a formula (i.e.\ sequent), $\enf{-}$,
suitable for establishing Theorem~\ref{thm:embed}, is given in
Figure~\ref{fig:norm}\footnote{This is a mathematical notation for the
  procedures formalised in \cite{formalisation} and also used in
  \cite{explog}.}. It maps any formula to an isomorphic formula from
the class of \emph{exp-log normal forms} ($\mathcal{E}$), defined by
the mutually inductively defined classes of base formulas
($\Base$), conjunctions ($\CNF$), and disjunctions
($\DNF$):
  \begin{align*}
    \Base\ni b &::= p ~|~ d & \\
    \CNF\ni c &::= (c_1\to b_1) \wedge \cdots \wedge (c_n\to b_n) & (n\ge 0)\ \\
    \DNF\ni d &::= c_1 \vee \cdots \vee c_n & (n\ge 2)\ \\
    \mathcal{E}\ni e &::= c ~|~ d, &
  \end{align*}
  that is,
  \begin{align*}
    \Base\ni b &::= p ~|~ d & \mathcal{E}\ni e &::= c ~|~ d\\
    \CNF\ni c &::= \prod_{i=1}^{n\ge 0} b_i^{c_i} & \DNF\ni d &::= \sum_{i=1}^{n\ge 2} c_i,
  \end{align*}
  where $p$ denotes a prime formula.  The variables $p, b, c, d, e$,
  possibly with indexing subscripts, will always be used to stand for
  members of the corresponding class.  The unit $1$ (i.e. the formula
  $\top$) is not a prime formula, but rather denotes the nullary
  product $\prod_{i=1}^0 b_i^{c_i}$.

  \begin{figure}
    \centering
    \begin{align*}
      \Base \ni b &::= p ~|~ d &
                                 \CNF \ni c &::= 1 ~|~ b^{c_1} c_2\\
      \DNF \ni d &::= c_1 \plus c_2 ~|~ c \plus d & 
                                                    \ENF \ni e &::= c ~|~ d
    \end{align*}
    
    \begin{align*}
      \nplus{-}{-} &: \ENF \to \ENF \to \DNF\\
      \nplus{c_1}{e_2} &:= c_1 \plus e_2\\
      \nplus{(c_{11} \plus c_{12})}{e_2} &:= c_{11}  \plus  (c_{12}  \plus  e_{2})\\
      \nplus{(c_{11} \plus d_{12})}{e_2} &:= c_{11} \plus (\nplus{d_{12}}{e_2})\\[1em]
      \ntimes{-}{-} &: \CNF \to \CNF \to \CNF\\
      \ntimes{1}{c_2} &:= c_2\\
      \ntimes{b^{c_{11}}c_{12}}{c_2} &:= b^{c_{11}} (\ntimes{c_{12}}{c_2})\\[1em]
      \distribone{-}{-} &: \CNF \to \ENF \to \ENF\\
      \distribone{c_1}{c_2} &:= \ntimes{c_1}{c_2}\\
      \distribone{c_1}{(c_{21} \plus c_{22})} &:= (\ntimes{c_1}{c_{21}})  \plus  (\ntimes{c_1}{c_{22}})\\
      \distribone{c_1}{(c_{21} \plus d_{22})} &:= (\ntimes{c_1}{c_{21}})  \plus  (\distribone{c_1}{d_{22}})\\[1em]
      \explogone{-}{-} &: \Base \to \ENF \to \CNF\\
      \explogone{b}{c} &:= b^c 1\\
      \explogone{b}{(c_1  \plus  c_2)} &:= \ntimes{(b^{c_1} 1)}{(b^{c_2} 1)}\\
      \explogone{b}{(c_1  \plus  d_2)} &:= \ntimes{(b^{c_1} 1)}{(\explogone{b}{d_2})}\\[1em]
      \explog{-}{-} &: \CNF \to \ENF \to \CNF\\
      \explog{1}{e_2} &:= 1\\
      \explog{b^{c_{11}} c_{12}}{e_2} &:= \ntimes{(\explogone{b}{(\distribone{c_{11}}{e_2})})}{(\explog{c_{12}}{e_2})}\\[1em]
      \distrib{-}{-} &: \ENF \to \ENF \to \ENF\\
      \distrib{c_1}{e_2} &:= \distribone{c_1}{e_2}\\
      \distrib{(c_{11} \plus c_{12})}{e_2} &:= \nplus{(\distribone{c_{11}}{e_2})}{(\distribone{c_{12}}{e_2})}\\
      \distrib{(c_{11} \plus d_{12})}{e_2} &:= \nplus{(\distribone{c_{11}}{e_2})}{(\distrib{d_{12}}{e_2})}
    \end{align*}
    
    \begin{align*}
      \enf{-} &: \Formula \to \ENF & \enfpos{-} &: \Formula \to \CNF\\
      \enf{p} &:= p^1 1 & \enfpos{p} &:= p^1 1\\
      \enf{F \vee G} &:= \nplus{\enf{F}}{\enf{G}} & \enfpos{F \vee G} &:= (\nplus{\enfpos{F}}{\enfpos{G}})^1 1\\
      \enf{F \wedge G} &:= \distrib{\enf{F}}{\enf{G}} & \enfpos{F \wedge G} &:= \ntimes{\enfpos{F}}{\enfpos{G}}\\
      \enf{F \to G} &:= \explog{\enfpos{G}}{\enf{F}} & \enfpos{F \to G} &:= \explog{\enfpos{G}}{\enf{F}}
    \end{align*}
    
    \caption{Formula normalization functions}
    \label{fig:norm}
  \end{figure}
  
  The formal definitions from Figure~\ref{fig:norm} are a mathematical
  notation for definitions implemented~\cite{formalisation} using the
  Coq proof assistant.  The function $\enf{-}$ is defined
  simultaneously with the $\enfpos{-}$-function, whose purpose is to
  guarantee the desired ordering of HSI's, i.e. that equation
  (\ref{hsi:distr}) is not applied at the base of exponentiation
  before equation (\ref{hsi:exp:2}). It uses operations
  $\nplus{}{}, \ntimes{}{}, \distrib{}{}, \distribone{}{},
  \explog{}{}$, and $\explogone{}{}$. The intuition behind these
  operations is as follows.

  The function $\nplus{}{}$ turns a binary plus (disjunction) into an
  $n$-ary one, more precisely, it flattens a tree of binary
  $\plus$-constructors 
  into a tail-inductive list, i.e. a list with all the $\plus$-s
  associated to the right,
  $c_1\plus(c_2\plus(c_3\plus(\cdots+c_n)))$. The function
  $\ntimes{}{}$ does the analogous thing for multiplication
  (conjunction).

  The functions $\distrib{}{}$ and $\distribone{}{}$ apply the left,
  correspondingly right, distributivity law. They are meant to
  implement the following informal equations:
  \begin{align*}
    \distrib{p}{d} &= \distribone{p^11}{d} & \distribone{c}{p} &= c p^11\\
    \distrib{\left(\sum_{i=1}^n c_i\right)}{d} &= \sum_{i=1}^n \left(\distribone{c_i}{d}\right) & \distribone{c}{\sum_{i=1}^{n} c_i} &= \sum_{i=1}^{n} c c_i.
  \end{align*}
  Finally, the functions $\explog{}{}$ and $\explogone{}{}$ normalise
  exponentiations following (\ref{hsi:exp:1}), (\ref{hsi:exp:2}) and
  (\ref{hsi:exp:3}), in order for the following informal equations to
  hold:
  \begin{align*}
    \explogone{b}{p} &= b^{p^1 1}1 & \explog{b^1 1}{e} &= \explogone{b}{e}\\
    \explogone{b}{\sum_{i=1}^{n} c_i} &= \prod_{i=1}^{n} {b}^{c_i} & \explog{\left(\prod_{i=1}^{n\ge 0} {b_i}^{c_i}\right)}{e} &= \prod_{i=1}^{n\ge 0} \left(\explogone{{b_i}}{\left(\distribone{c_i}{e}\right)}\right).
  \end{align*}

  In order to prove the main theorem of this section,
  Theorem~\ref{thm:embed}, we shall need to establish
  Lemma~\ref{lem:equations} on \emph{formal} equalities, where ``$=$''
  denotes multiset equality of terms, or, more precisely, we will
  have the formal equality
  $b_1^{c_1}\cdots b_i^{c_i}\cdots b_j^{c_j}\cdots b_n^{c_n} =
  b_1^{c_1}\cdots b_j^{c_j}\cdots b_i^{c_i}\cdots b_n^{c_n}$, for any
  $1\leq i,j\leq n$. We use multiset equality because it allows us to
  work modulo commutativity and associativity of multiplication
  (conjunction), and as it is in common usage when presenting sequent calculi.

  But, we first point out that the normalization function $\enf{-}$
  provides a way to classify propositional formulas while preserving
  isomorphism.
  \begin{theorem}\label{thm:propositional:hierarchy}
    For every propositional formula $F$, there is an isomorphic
    formula $e\in \DNF\cup\CNF$, where the classes
    $\DNF$, $\CNF$, and $\Base$ are defined
    simultaneously as follows,
    \begin{align*}
      \DNF\ni d &::= c_1 \vee \cdots \vee c_n & (n\ge 2)\ \\
      \CNF\ni c &::= (c_1\to b_1) \wedge \cdots \wedge (c_n\to b_n) & (n\ge 0)\\
      \Base\ni b &::= p ~|~ d,
    \end{align*}
    or, in exponential polynomial notation:
    \begin{align*}
      \DNF\ni d &::= \sum_{i=1}^{n\ge 2} c_i & 
      \CNF\ni c &::= \prod_{i=1}^{n\ge 0} b_i^{c_i} &
      \Base\ni b &::= p ~|~ d,
    \end{align*}
    where $p$ denotes a prime formula.
  \end{theorem}
  \begin{proof} Given $F$, the required $e$ is $\enf{F}$.  The
    normalization procedure defined in Figure~\ref{fig:norm} is
    structurally recursive, hence terminating, and with range
    $\DNF\cup\CNF$.

    Isomorphism between $F$ and $\enf{F}$ holds as a consequence
    of~(\ref{eq:hsivalid}), because the equations defining the
    normalization procedure in Figure~\ref{fig:norm} are just
    instances or simple consequences of the HSI: to see this, it
    suffices to replace the symbols $\times$, $\ltimes$ and $\rtimes$
    by multiplication, $\oplus$ by addition, $\explogone{b}{e}$ by
    $b^e$, $\explog{c}{e}$ by $c^e$, and $\enf{F}$ and $\enfpos{F}$ by
    $F$.
  \end{proof}
  
  \begin{lemma}\label{lem:equations} The following equations hold for
    the normalization functions defined in Figure~\ref{fig:norm}:
    \begin{align}
      \ntimes{c}{1} &= c \label{ntimes_top}\\
      \ntimes{c_1}{(\ntimes{c_2}{c_3})} &= \ntimes{(\ntimes{c_1}{c_2})}{c_3} \label{ntimes_assoc}\\
      \nplusone{d}{(\nplus{e_2}{e_3})} &= \nplusone{(\nplusone{d}{e_2})}{e_3} \label{nplus1_assoc}\\
      \nplus{e_1}{(\nplus{e_2}{e_3})} &= \nplus{(\nplus{e_1}{e_2})}{e_3} \label{nplus_assoc}\\
      \distribzero{c}{(\nplusone{d}{e})} &= \nplus{(\distribzero{c}{d})}{(\distribone{c}{e})} \label{distrib0_nplus1}\\
      \distribzero{c}{(\nplus{e_1}{e_2})} &= \nplus{(\distribone{c}{e_1})}{(\distribone{c}{e_2})} \label{distrib0_nplus}\\
      \distribn{(\nplusone{d}{e_1})}{e_2} &= \nplus{(\distribn{d}{e_2})}{(\distrib{e_1}{e_2})} \label{distribn_nplus1}\\
      \distribn{(\nplus{e_0}{e_1})}{e_2} &= \nplus{(\distrib{e_0}{e_2})}{(\distrib{e_1}{e_2})} \label{distribn_nplus}\\
      \distribone{1}{e} &= e \label{distrib1_top}\\
      \distribone{c_1}{(\distribzero{c_2}{d})} &= \distribzero{(\ntimes{c_1}{c_2})}{d} \label{distrib1_distrib0}\\
      \distribone{c_1}{(\distribone{c_2}{e})} &= \distribone{(\ntimes{c_1}{c_2})}{e} \label{distrib1_distrib1}\\
      \distribone{c}{(\distribn{d}{e})} &= \distrib{(\distribzero{c}{d})}{e} \label{distrib1_distribn}\\
      \distribone{c}{(\distrib{e_1}{e_2})} &= \distrib{(\distribone{c}{e_1})}{e_2} \label{distrib1_distrib}\\
      \distribn{d}{(\distrib{e_1}{e_2})} &= \distrib{(\distribn{d}{e_1})}{e_2} \label{distribn_distrib}\\
      \distrib{e_1}{(\distrib{e_2}{e_3})} &= \distrib{(\distrib{e_1}{e_2})}{e_3} \label{distrib_assoc}\\
      \explog{c}{1} &= c \label{explogn_top}\\
      \explog{(\ntimes{c_1}{c_2})}{e} &= \ntimes{(\explog{c_1}{e})}{(\explog{c_2}{e})}\label{explogn_ntimes}\\
      \explogzero{b}{(\nplus{d}{e})} &= \ntimes{(\explogzero{b}{d})}{(\explogone{b}{e})} \label{explog0_nplus1}\\
      \explogzero{b}{(\nplus{e_1}{e_2})} &= \ntimes{(\explogone{b}{e_1})}{(\explogone{b}{e_2})} \label{explog0_nplus}\\
      \explogone{b}{(\distrib{e_1}{e_2})} &= \explog{(\explogone{b}{e_1})}{e_2} \label{explogn_explog1}\\
      \explog{c}{(\distrib{e_1}{e_2})} &= \explog{(\explog{c}{e_1})}{e_2} \label{explogn_distrib}\\
      \explog{c}{(\nplus{e_1}{e_2})} &= \ntimes{(\explog{c}{e_1})}{(\explog{c}{e_2})} \label{explogn_nplus}\\
      \explog{c}{(\distribn{(\nplus{e_1}{e_2})}{e_3})} &= \ntimes{(\explog{c}{(\distrib{e_1}{e_3})})}{(\explog{c}{(\distrib{e_2}{e_3})})} \label{explogn_distribn_nplus}
    \end{align}
  \end{lemma}
  \begin{proof} The proofs proceed as follows: (\ref{ntimes_top}) by
    induction on $c$; (\ref{ntimes_assoc}) by induction on $c_1$;
    (\ref{nplus1_assoc}) by induction on $d$; (\ref{nplus_assoc}) by
    case analysis on $e_1, e_2, e_3$ and using (\ref{nplus1_assoc});
    (\ref{distrib0_nplus1}) by induction on $d$;
    (\ref{distrib0_nplus}) by case analysis on $e_1, e_2$ and using
    (\ref{distrib0_nplus1}); (\ref{distribn_nplus1}) by induction on
    $d$ and using (\ref{nplus_assoc}); (\ref{distribn_nplus}) by case
    analysis on $e_1, e_2$ and using (\ref{distribn_nplus1});
    (\ref{distrib1_top}) by induction on $e$;
    (\ref{distrib1_distrib0}) by induction on $d$ and using
    (\ref{nplus_assoc}); (\ref{distrib1_distrib1}) by case analysis on
    $e$ and using (\ref{distrib1_distrib0}) and (\ref{nplus_assoc});
    (\ref{distrib1_distribn}) by induction on $d$ and using
    (\ref{distrib0_nplus}) and (\ref{distrib1_distrib1});
    (\ref{distrib1_distrib}) by case analysis on $e_1$ and using
    (\ref{distrib1_distribn}) and (\ref{distrib1_distrib1});
    (\ref{distribn_distrib}) by induction on $d$ and using
    (\ref{distribn_nplus}) and (\ref{distrib1_distrib});
    (\ref{distrib_assoc}) by case analysis on $e_1$ and using
    (\ref{distrib1_distrib}) and (\ref{distribn_distrib});
    (\ref{explogn_top}) by induction on $c$ and using
    (\ref{ntimes_top}); (\ref{explogn_ntimes}) by induction on $c_1$
    and using (\ref{ntimes_assoc}); (\ref{explog0_nplus1}) by
    induction on $d$; (\ref{explog0_nplus}) by induction on $e_1$ and
    using (\ref{explog0_nplus1}); (\ref{explogn_explog1}) by induction
    on $e_1$ and using (\ref{ntimes_top}) and (\ref{explog0_nplus});
    (\ref{explogn_distrib}) by induction on $c$ and using
    (\ref{explogn_ntimes}), (\ref{explogn_explog1}), and
    (\ref{distrib1_distrib}); (\ref{explogn_nplus}) by induction on
    $c$ and using (\ref{distrib0_nplus}) and (\ref{explog0_nplus});
    (\ref{explogn_distribn_nplus}) by using (\ref{distribn_nplus}) and
    (\ref{explogn_nplus}).

    It may be interesting to notice that in fact (\ref{explogn_nplus})
    and (\ref{explogn_distribn_nplus}) are the only cases where the
    multiset nature of the equality is needed -- it is needed because
    it implies the associativity and commutativity of
    ``$\times$''. For all of the other cases, it suffices to have the
    definitional (intensional) equality on ordered sequences (lists),
    which is the equality that we worked with in the Coq
    formalisation. The remark is interesting for us, because it
    pinpoints which sequent calculi proof rules exactly need to work
    with multisets: as can be seen from the proof of
    Theorem~\ref{thm:embed}, these rules are ($\vee_l$) and
    ($\to_l^{\vee}$).
  \end{proof}

  \begin{figure}
    \centering
    \begin{gather}
      \tag{axiom} \myinfer{\explogone{p}{(\distribone{p}{e}})}{} \\[1em]
      \tag{$\vee^1_r$} \myinfer{\explogone{(c_1 \plus c_2)}{e}}{\explog{c_1}{e}} \\[1em]
      \tag{$\vee^2_r$} \myinfer{\explogone{(c_1 \plus c_2)}{e}}{\explog{c_2}{e}} \\[1em]
      \tag{$\to_l^{P}$} \myinfer{\explog{c}{(\distribone{(\explog{\enfpos{F}}{p})}{(\distribone{p}{e})})}}{\explog{c}{(\distrib{\enf{F}}{(\distribone{p}{e})})}} \\[1em]
      \tag{$\to_l^{\to}$} \myinfer{\explog{c}{(\distribone{(\explog{\enfpos{H}}{(\explog{\enfpos{G}}{e_1})})}{e_2})}}{\ntimes{(\explog{(\explog{\enfpos{G}}{e_1})}{(\distribone{(\explog{\enfpos{H}}{\enf{G}})}{e_2})})}{(\explog{c}{(\distrib{\enf{H}}{e_2})})}}
    \end{gather}
    \caption{Proof rules of the High-school sequent calculus (HS) for G4ip}
    \label{fig:hs}
  \end{figure}

  Armed with a precise and terminating transformation of formulas, we
  can now state the HS variant of G4ip in Figure~\ref{fig:hs}. For
  prime formulas $p$, we make a harmless abuse of notation by writing
  just $p$ instead of $p^11$ when an argument of type $\CNF$ is
  expected by an operation.

  Notice that our calculus consists of non-invertible rules only,
  tagged with the tag of the G4ip rule they correspond to (to be shown
  in Theorem~\ref{thm:embed} below). The rules ($\to_l^{P}$) and
  ($\to_l^{\to}$) mention usual formulas $F,G,H$. This is done on
  purpose, so that the correspondence to G4ip rules is as tight as
  possible. If one wants to mention only formulas from the normalised
  classes, one can consider the following reformulations of the rules,
  \begin{gather}
    \tag{$\to_l^{P}$'} \myinfer{\explog{c}{(\distribone{(\explog{c_0}{p})}{(\distribone{p}{e})})}}{\explog{c}{(\distrib{\partial c_0}{(\distribone{p}{e})})}} \\[1em]
    \tag{$\to_l^{\to}$'} \myinfer{\explog{c}{(\distribone{(\explog{c_1}{(\explog{c_2}{e_1})})}{e_2})}}{\ntimes{(\explog{(\explog{c_2}{e_1})}{(\distribone{(\explog{c_1}{\partial c_2})}{e_2})})}{(\explog{c}{(\distrib{\partial c_1}{e_2})})}},
  \end{gather}
  where $\partial$ denotes the map $\enfpos{F}\mapsto\enf{F}$ that
  distributes the product over the sums of the form
  $(c_1+\cdots+c_n)^1$ in $\enfpos{F}$; here, the exponent $1$ is used
  to suspend normalization, that is, permit the isomorphism
  (\ref{hsi:exp:2}) to be applied before (\ref{hsi:distr}) at the base
  of exponentiation.

  Note also, that one can make the normalization functions disappear
  from any \emph{concrete} proof in HS notation, so one could adopt
  the view that the functions are there merely for a compact
  presentation of the rules. For instance, considering the case of the
  $(\to^\to_l)$'-rule from HS, where $c:=p^11$, $c_1:=p^11$,
  $c_2:=q^11$, $e_1:=r^11$, $e_2:=s^11$, and $p,q,r,s$ are prime
  formulas, we retrieve just the corresponding G4ip rule in polynomial
  notation,
  \begin{gather*}
    \myinfer{p^{p^{q^{r^1 1}1} s^1 1}1}{\left(q^{r^1 p^{q^1 1} s^1 1}1\right)\times\left(p^{p^1 s^1 1} 1\right)},
  \end{gather*}
  or the more readable one, by the harmless abuse of notation for
  prime formulas mentioned earlier ($c:=p$, $c_1:=p$, $c_2:=q$,
  $e_1:=r$, $e_2:=s$) and by omitting the trailing $1$:
  \begin{gather*}
    \myinfer{p^{p^{q^{r}}s}}{q^{r p^{q} s}p^{p s}}.
  \end{gather*}
  For a concrete example of the rules involving disjunction, consider
  the case where $c_1:=p$, $c_2:=q$, $e:=r + s$, and the concrete
  instance of the ($\vee^1_r$) rule:
  \begin{gather*}
    \myinfer{\left(p+q\right)^{r}\left(p+q\right)^{s}}{p^{r} p^{s}}.
  \end{gather*}

  We now show in which sense HS is a version of G4ip. This will
  also imply that HS is a proof system complete for intuitionistic
  provability.

  \begin{theorem}\label{thm:embed} 
    Every derivation of $F$ in G4ip can be transformed to a derivation of
    $\enf{F}$ in HS.
  \end{theorem}

  \begin{proof} The proof is by induction on the derivation. The
    premises are connected by ``$\times$''. At each case, we first
    apply (\ref{ntimes_top}) and (\ref{distrib1_top}) of
    Lemma~\ref{lem:equations} to slightly simplify the involved
    expressions. Then, the non-invertible rules, (axiom),
    ($\vee_r^1$), ($\vee_r^2$), ($\to_l^{P}$), and ($\to_l^{\to}$),
    are directly proven by their HS correspondent rule. As for the
    invertible rules, there is no need to use HS rules to interpret
    them, because applying the normalization function $\enf{-}$ on
    formulas is enough, more precisely:
    \begin{itemize}
    \item ($\to_r$) is proven by (\ref{explogn_distrib});
    \item ($\wedge_r$) is proven by (\ref{explogn_ntimes});
    \item ($\wedge_l$) is proven by (\ref{distrib_assoc});
    \item ($\vee_l$) is proven by (\ref{explogn_distribn_nplus});
    \item ($\to_l^{\wedge}$) is proven also by (\ref{explogn_distrib});
    \item and, ($\to_l^{\vee}$) is proven by (\ref{explogn_nplus}) (and (\ref{distrib_assoc})).
    \end{itemize}
  \end{proof}
  
  Although it may appear to be complex to transform G4ip proofs to HS
  proofs, when one wants to formally apply the normalization functions
  of Figure~\ref{fig:norm}, this transformation is actually quite easy
  and can be efficiently also performed by hand using high-school
  arithmetic. We give an example to show how it works.

  \begin{example}
    The following G4ip derivation of
    $r \wedge (q\to (r\vee t)\to s) \to q\to s$,
    \[
      \infer[(\to_r)]{\left(s^q\right)^{r \left(s^{r+t}\right)^q}}{
        \infer[(\to^P_l)]{s^{q r \left(s^{r+t}\right)^q}}{
          \infer[(\to^\vee_l)]{s^{q r s^{r+t}}}{
            \infer[(\to^P_l)]{s^{q r s^r s^t}}{
              \infer[\text{axiom}]{s^{q r s s^t}}{}
            }
          }
        }
      }
    \]
    is mapped to the following HS derivation:
    \[
        \infer[(\to^P_l)]{s^{q r s^{rq} s^{tq}}}{
            \infer[(\to^P_l)]{s^{q r s^r s^t}}{
              \infer[\text{axiom}]{s^{q r s s^t}}{}
          }
        }
      .
    \]
  \end{example}

  We end this section by establishing in the following
  Theorem~\ref{thm:hs2lj} that HS can be seen as a \emph{fragment} of
  G4ip, one that is complete for provability even if it does not
  contain invertible proof rules.  As a corollary, by composition of
  Theorem~\ref{thm:embed} and Theorem~\ref{thm:hs2lj}, one gets that
  all G4ip proofs can be mapped into a fragment of G4ip proofs, this
  mapping being proof-identity-preserving.
  \begin{theorem}\label{thm:hs2lj}
    Suppose that HS proves the formula $c$. 

    Let $\overline{a}$ denote the formula that is obtained from an
    expression $a$ using the symbols
    $\{\ltimes, \rtimes, \plus, \enf{\cdot}, \enfpos{\cdot},
    \explogone{}{}, \explog{}{}\}$, by replacing $\ltimes$ and
    $\rtimes$ by $\wedge$, $\plus$ by $\vee$, $\enf{I}$ and
    $\enfpos{I}$ by $I$, $\explogone{j}{k}$ by $k\to j$, and
    $\explog{j}{k}$ by $k\to j$.
    
    Then G4ip proves the formula $\overline{c}$.
  \end{theorem}
  \begin{proof}
    This can be proved formally by induction on the derivation, but
    the translation is straightforward: one keeps in G4ip essentially
    the same tree of proof rules from HS, but replaces the premise $p$
    and conclusion $c$ by the premise $\overline{p}$ and conclusion
    $\overline{c}$.
    
    That means in particular that, when $c$ is in one of the possible
    forms $\explogone{p}{(\distribone{p}{e}})$, or
    $\explogone{(c_1 \plus c_2)}{e}$, or
    $\explog{c}{(\distribone{(\explog{\enfpos{F}}{p})}{(\distribone{p}{e})})}$,
    or
    $\explog{c}{(\distribone{(\explog{\enfpos{H}}{(\explog{\enfpos{G}}{e_1})})}{e_2})}$,
    then $\overline{c}$ is correspondingly of form
    ${p}\wedge{\overline{e}}\to{p}$, or
    ${\overline{e}}\to{(\overline{c_1} \vee \overline{c_2})}$, or
    ${{({p}\to{F})}\wedge{({p}\wedge{\overline{e}})}}\to\overline{c}$,
    or
    ${((\overline{e_1}\to G)\to
      H)\wedge{\overline{e_2}}}\to\overline{c}$.
  \end{proof}
  

\section{The Inequality Interpretation of Proof Rules}
\label{sec:inequality}
\allowdisplaybreaks
\def\interp#1{[\![#1]\!]}
\def\squigto{\quad\rightsquigarrow\quad}
In this section, we show that the inference rules for G4ip 
can be interpreted as inequalities relating the exponential
polynomials corresponding to the premises and the conclusion. This extends the previous observation that invertible rules are equalities.

We start by exploring the inequality interpretation of G4ip, in order
to keep the presentation somewhat simple.
The main result is the following.
\begin{theorem}\label{thm:ruledecreasing}
  Let $\mathcal R$ be an inference rule of G4ip.
  If the variables $F,G,H,I,P$ are interpreted to be natural numbers
  strictly greater than $1$, then the value of the premise of the rule
  is less than or equal to the value of the conclusion. Moreover, the inequality
  is strict if and only if $\mathcal R$ is not invertible.
\end{theorem}
Note that for inference rules with multiple premises, we take the interpretation
to be the \emph{product} of the interpretations of the individual premises.
\begin{remark}
  We should clarify that our notion of \emph{invertibility} differs slightly
  from that of Dyckhoff. In \cite{Dyckhoff1992}, the $(\to_l^P)$ rule is considered
  invertible, as any proof of $(P\to F), P, \Gamma \vdash G$ may be transformed into a proof
  of $F, P, \Gamma \vdash G$. From our point of view, however, the $(\to_l^P)$ rule is
  \emph{not} invertible, because we consider invertibility to be a property of
  \emph{formulas} rather than rules. Given the sequent $(P\to F), \Gamma \vdash G$, we
  \emph{may} be able to decompose $P\to F$ immediately, or we may not, depending
  on whether $P$ is present in the context $\Gamma$. Because we cannot guarantee that
  the formula can be decomposed, we consider the rule to be non-invertible.

  A different way of seeing this is to consider the following formulation of the
  $(\to_l^P)$ rule:
  \begin{align*}
      \tag{${\to_l^P}'$} \myinfer{(P\to F),\Gamma\vdash G}{F,\Gamma\vdash G\qquad P\in\Gamma}
  \end{align*}

  This rule is equivalent to the usual rule, and it is straightforward to
  convert between proofs using the former rule and proofs using the latter.
  Note, however, that this rule is \emph{not} invertible in the traditional
  sense. While a proof of the conclusion certainly implies that there exists a
  proof of the first premise, this is not the case for the second premise.
\end{remark}

To present the inequality interpretation, we define an interpretation function $\interp{-}$ that maps
formulas and contexts to natural
numbers. The function is defined as follows:

\begin{align*}
  \interp{F\vee G}&=\interp{F}+\interp{G}&
  \interp{F\wedge G}&=\interp{F}\cdot\interp{G}\\
  \interp{G\to F}&=\interp{F}^{\interp{G}}&
  \interp{\Gamma,F}&=\interp{\Gamma}\cdot\interp{F}\\
  \interp{P}&=2 &\interp{\top}&=\interp{\cdot}=1
\end{align*}
Note that with this interpretation, we have the following property:
\begin{lemma}\label{lem:one-is-top}
  If $\interp{F}=1$ then $F\cong \top$.
\end{lemma}
\begin{proof}
  By a straightforward induction on the structure of formulas.\qedhere
\end{proof}
This observation justifies our assumption that when we apply an
inference rule, the variables involved must have a value greater than
or equal to $2$. For instance, consider the $(\to_l^\to)$ rule. Here, we could have $F,G,H=P$,
$\Gamma=\cdot$, and $I=\top$, which would result in the following
interpretation:
\[
\myinfer{1^{2^{2^2}\cdot1}}{(2^2)^{2^2\cdot1}\cdot1^{2\cdot1}}
\]
Clearly, the premise has a greater value than the conclusion, hence the
inequality interpretation does not work, unless we assume all the
variables (except the one corresponding to $\Gamma$) have values not
less than $2$.

Note, however, that this is an entirely reasonable assumption given
the content of Lemma~\ref{lem:one-is-top}. If $I=\top$, there is no
reason to apply the $(\to_l^\to)$ rule, as we already know $\top$ is
provable\footnote{In fact, modern presentations of G4ip \cite{troelstra2000basic} omit $\top$ as
a formula entirely, as it --- from a proof search perspective --- is
completely superfluous.}. Thus, we assume that the formulas and contexts in
question have been subjected to the following simplification rules
first:
\begin{align*}
  \top\wedge F &\squigto F&F\wedge\top&\squigto F\\
  \top\to F&\squigto F&F\to\top&\squigto\top\\
\Gamma,\top&\squigto\Gamma
\end{align*}
These simplifications correspond to the high-school
identities~\eqref{hsi:multunit}, \eqref{hsi:powerone},
and~\eqref{hsi:onepower}.
Note that we do not need to reduce occurrences of $\top$ inside
disjunctions, as $\interp{F\vee G}\ge2$ for all formulas $F,G$. This
leads to the following lemma:
\begin{lemma}
  For all formulas $F$, either $\interp{F}\ge2$ or
  $F\rightsquigarrow^*\top$.
\end{lemma}
\begin{proof}
  By induction on the structure of $F$. We show here a representative
  case. If $F=G\to H$, we apply the induction hypothesis to $H$. If
  $H\rightsquigarrow^*\top$, then $F\rightsquigarrow^*\top$ by the
  definition of $\rightsquigarrow$. If not, we have $\interp{H}\ge2$,
  and thus
  \[\interp{G\to H} = \interp{H}^{\interp{G}}\ge2^{\interp{G}}\ge2,\]
  by using the fact that $\interp{H}\ge2$ and $\interp{G}\ge1$
  respectively.
\end{proof}
Alternatively, one can simply replace all occurrences of $\top$ with
any formula with a unique proof, such as $P\to P$ for some fresh
atomic formula $P$.

In the rest of this section, we will omit the interpretation function,
as it will be obvious from the context whether we are talking about
the formula or the interpretation to which it is mapped.

Let us now return to the inference rules of G4ip. The fact that the
invertible rules preserve the value of the sequents is immediate by
inspection of the inference rules. For instance, for the $(\vee_l)$
rule, we would need to show that 
\[H^{(F+G)\Gamma}=H^{F\Gamma}H^{G\Gamma},\]
but this is a simple arithmetical equality. This leaves the matter of
establishing the non-invertible rules as strict inequalities. For the
$(\vee_r^1)$ and $(\vee_r^2)$ rules, this is immediate, as $F+G>F$ and
$F+G>G$ whenever $F,G\ge1$.

For the $(\to_l^P)$ rule, since we do not apply rules that do nothing
after the simplification, we have that $F\not\rightsquigarrow^*\top$ and thus $F\ge2$.
As $P$ is an atomic formula, its interpretation is $2$, and thus
\begin{align*}
  F^P&=F^2 >F&&\text{and hence}\\
  G^{F^PP\Gamma}&>G^{FP\Gamma}&&\text{by monotonicity.}
\end{align*}

This leaves the inequality associated to the $(\to_l^\to)$ rule, for
which we will need a few lemmas first:
\begin{lemma}
  If the inequality $2^{H^{G^F}-H} > G^{FH^G}$ holds for all
  $G,H,F\ge2$, then $I^{H^{G^F}\Gamma} >
  (G^F)^{H^G\Gamma}I^{H\Gamma}$ for all $F,G,H,I\ge2$ and $\Gamma\ge1$.
\end{lemma}
\begin{proof}
We reason as follows:
\begin{align*}
  2^{H^{G^F}-H} &> G^{FH^G} &&\text{by assumption.}\\
  I^{H^{G^F}-H} &> G^{FH^G} &&\text{as $I\ge2$.}\\ 
  I^{(H^{G^F}-H)\Gamma} &> G^{FH^G\Gamma} &&\text{by raising each side to the power $\Gamma$.}\\ 
  I^{H^{G^F}\Gamma-H\Gamma} &> G^{FH^G\Gamma} &&\text{by distributivity.}\\
  I^{H^{G^F}\Gamma} &> G^{FH^G\Gamma}I^{H\Gamma}&&\text{by multiplying with $I^{H\Gamma}$.}\\
  I^{H^{G^F}\Gamma} &> (G^F)^{H^G\Gamma}I^{H\Gamma}&&\text{by the high-school identity.}
\end{align*}
\qedhere
\end{proof}
Next, we need a few further lemmas in order to discharge the
assumption in the preceding lemma:
\begin{lemma}\label{monotone} The following inequalities hold for $G\ge 2$:
  \begin{gather*}
    \forall F\ge 1\left(2^{G^{F+1}}-2^{G^F}\ge 2^{G^2}-2^{G}\right)\\
    \forall F\ge 2\left(2^{G^{F+1}}-2^{G^F}\ge 2^{G^3}-2^{G^2}\right).
  \end{gather*}
\end{lemma}
\begin{proof}
  Both statements are proved by induction on $F$, and for both
  induction cases it is enough to have the following inequality, for
  $F\ge 1$:
  \[
    2^{G^{F+2}}-2^{G^{F+1}}\ge 2^{G^{F+1}}-2^{G^F},
  \]
  that is,
  \[
    2^{G^{F+2}}\ge 2^{G^{F+1}+1}-2^{G^F}.
  \]
  We can actually prove the stronger statement
  \[
    2^{G^{F+2}}\ge 2^{G^{F+1}+1}
  \]
  by using the monotonicity of the $\log_2$ function, since
  $\log_2{2^{G^{F+2}}} = G^{F+2}$,
  $\log_2{2^{G^{F+1}+1}} = G^{F+1}+1$, and because
  \[
    G^{F+2}\ge G^{F+1}+1
  \]
  clearly holds when $G\ge 2$.
\end{proof}
\begin{lemma}
Given $F,H\ge2$ and $G\ge3$, the following inequalities hold:
\begin{equation}
  \label{eq:gfggegf1}
  G^F-G-1\ge G^{F-1}
\end{equation}
\begin{equation}
  \label{eq:2gf1gefg}
  2^{G^{F-1}}\ge FG
\end{equation}
\begin{equation}
  \label{eq:add1trick}
  FH^GG\ge FH^{G-1}G+1
\end{equation}
\end{lemma}
\begin{proof}
  To prove~\eqref{eq:gfggegf1}, we reason as follows:
\begin{align*}
  G^{F-2}&\ge G^{2-2}=G^0=1&&\text{as $F\ge 2$.}\\
  3G^{F-2}-1&>G^{F-2}&&\text{as $3n-1>n$ when $n\ge1$.}\\
  G^{F-1}-1&>G^{F-2}&&\text{as $G\ge3$.}\\
  G^F-G &> G^{F-1} &&\text{by multiplying with $G$.}\\
  G^F-G &\ge G^{F-1}+1 &&\text{as $G,F\in\mathbb{N}$.}\\
  G^F-G-1 &\ge G^{F-1}&&\text{by rearranging.}  
\end{align*}
Next, to prove~\eqref{eq:2gf1gefg}, we do this in two steps: First,
we note that the inequality holds when $F=2$. To show this, we need to
show that $2^{G^{2-1}}=2^G\ge 2G$, that is $2^{G-1}\ge G$, which is
clear when $G\ge2$. Next, we observe that if the inequality holds for
some $F$, then it also holds with $F+1$ substituted in place of $F$.
For the right hand side of the inequality, this gives a difference of
$(F+1)G-FG=G$. For the left hand side, we reason as follows:
\begin{align*}
  2^{G^F}-2^{G^{F-1}}&\ge 2^{G^2}-2^{G^{2-1}}&&\text{as $F\ge 2$ and using Lemma~\ref{monotone}.}\\
  &=(2^G)^G-2^G &&\text{by the high-school identities.}\\
  &\ge(2^G)^2-2^G&&\text{as $G\ge2$.}\\
  &=2^G(2^G-1)\\
  &\ge G&&\text{as $2^G\ge G$ and $2^G-1\ge 1$.}
\end{align*}
As we have now established that 
\[2^{G^F}-2^{G^{F-1}}\ge (F+1)G-FG\]
for all $F\ge2$, the desired result follows from a straightforward
induction on $F$. 

Finally, to establish~\eqref{eq:add1trick}, we reason as follows:
\begin{align*}
  FH^{G-1}G(H-1)&\ge 1&&\text{as $F,G,H\ge2$.}\\
  FH^GG&\ge FH^{G-1}G+1&&\text{by rearranging.}
\end{align*}
\end{proof}

We can now establish the final lemma:
\begin{lemma}
  For all $F,G,H\ge2$, we have $2^{H^{G^F}-H} > G^{FH^G}$.
\end{lemma}
\begin{proof}
  We first prove this in the case where $G\ge3$:
  \begin{align*}
    2^{G^{F-1}}&\ge FG&&\text{by~\eqref{eq:2gf1gefg}.}\\ 
    2^{G^F-G-1}&\ge FG&&\text{as $G^F-G-1\ge G^{F-1}$ by~\eqref{eq:gfggegf1}.}\\
    H^{G^F-G-1}&\ge FG&&\text{as $H\ge2$.}\\
    H^{G^F-1}&\ge FH^GG&&\text{by multiplying with $H^G$.}\\
    H^{G^F-1}&\ge FH^{G-1}G+1&&\text{by~\eqref{eq:add1trick} and transitivity.}\\
    H^{G^F-1}-1&\ge FH^{G-1}G\\
    H^{G^F}-H&\ge FH^GG&&\text{by multiplying with $H$.}\\
    2^{H^{G^F}-H}&\ge 2^{FH^GG}&&\text{by monotonicity.}\\
    2^{H^{G^F}-H}&\ge (2^G)^{FH^G}&&\text{by the high-school identity.}\\
    2^{H^{G^F}-H}&> G^{FH^G}&&\text{as $2^G>G$ when $G\ge 2$.}
  \end{align*}
This takes care of the case when $G\ge3$. In the case when $G=2$, we
need to show the following strict inequality:
\[2^{H^{2^F}-H}>2^{FH^2}\]
First, we will establish the inequality
\[2^{2^F-2}-F>1\]
To do so, we note that it holds when $F=2$, and all that is needed,
then, is to establish that the expression $2^{2^F-2}-F$ is monotonic
in $F$ for all $F\ge 2$. Looking at successive differences, we get
\begin{align*}
  \left(2^{2^{F+1}-2}-(F+1)\right)-\left(2^{2^F-2}-F  \right)&=
  2^{2^{F+1}-2}-2^{2^F-2}-1\\
  &=(2^{2^{F+1}}-2^{2^F})\cdot2^{-2}-1\\
  &\ge(2^{2^3}-2^{2^2})\cdot2^{-2}-1 & \text{ by Lemma~\ref{monotone}}\\
  &=2^6-2^2-1>0
\end{align*}
whence the expression is monotonic in $F$. 
We can now complete the argument as follows:
\begin{align*}
  2^{2^F-2}-F&>1&&\text{by the preceding argument.}\\
  H^{2^F-2}-F&>1&&\text{as $H\ge2$.}\\
  H(H^{2^F-2}-F)&>1&&\text{by multiplying with $H$.}\\
  H^{2^F-1}-FH&>1&&\text{by simplification.}\\  
  H^{2^F-1}-1&>FH&&\text{by rearranging.}\\
  H^{2^F}-H&>FH^2&&\text{by multiplying with $H$.}\\
  2^{H^{2^F}-H}&> 2^{FH^2}&&\text{by monotonicity.}
\end{align*}
This concludes the proof.
\end{proof}
Combining the above lemmas, we now get
Theorem~\ref{thm:ruledecreasing} as a straightforward consequence.

Using this theorem, we can prove as an easy corollary that proof
search using the rules of G4ip is terminating. Because none of the
rules increases the value of the interpretation of sequents, and
because the non-invertible rules \emph{strictly decrease} this value,
it follows that the number of non-invertible rules in a derivation is
bounded by a function of the value of the goal sequent. Moreover, one
can show that there is at most a finite number of \emph{invertible}
rules between any two non-invertible rules in a derivation, and the
termination of proof search follows easily. Note that, even in the
case $\bot$ was considered, because the rule for $\bot$ does not have
any premises, its impact on termination would be trivial.

The traditional way of showing G4ip is terminating is also done by
assigning a measure to each sequent, but for this, it is sufficient to
show that the measure decreases along any \emph{branch} of the proof
tree. In our presentation, we have blurred the distinction between the
meta-level conjunction (i.e. multiple premises) and that of the object
level, as motivated by the corresponding equations for exponential
polynomials. 

It is natural to consider whether a similar approach would suffice to show the
termination of other calculi, such as the G3ip calculus. In this calculus, the
$\to_l$ rule has the following form:
\[
\myinfer{(A\to B), \Gamma \vdash C}{(A\to B), \Gamma \vdash A\qquad B, \Gamma \vdash C}
\]
If we assume $a$, $b$, $c$, and $\gamma$ are the interpretations of $A$, $B$, $C$,
and $\Gamma$ respectively, then the inequality interpretation would require that 
\[
c^{b^a\cdot\gamma}>a^{b^a\cdot\gamma}\cdot c^{b\cdot \gamma}
\]
Even if we ignore the second premise, in order to have $c^{b^a\cdot\gamma}$ be strictly
greater than $a^{b^a\cdot\gamma}$, we must have that $c$ is strictly greater than $a$.
If $A$ and $C$ are both atomic, this at the very least means that their
interpretations must be different, and raises the question of how one would even
determine which values should be assigned to the atomic formulas.

However, even a non-uniform interpretation of the atomic formulas would not
suffice because derivations in G3ip may \emph{loop}: in the presence of an
assumption of the form $A\to A$ (or in a more non-trivial case, $A\to B$ and $B\to A$)
a derivation of $\Gamma \to A$ may after a few applications of the $\to_l$ rule return to
the exact same sequent $\Gamma\to A$. In this case, the inequality interpretation
cannot possibly hold.

Finally, let us briefly remark on how to extend the above result to
the HS sequent calculus. The first step is to note that the
normalization functions shown in Figure~\ref{fig:norm} all preserve
the value of the interpretation. Thus, all that is needed is to show
that the inference rules must strictly decrease the associated values.
In this case, however, the necessary inequalities are exactly the ones
we established previously, and as the HS sequent calculus only has
non-invertible rules, termination of proof search is immediate. Note that this again
requires all occurrences of $\top$ to have been simplified away.


\section{An Intuitionistic Arithmetical Hierarchy}
\label{sec:hierarchy}

In classical first-order logic, every formula is equivalent to a
formula in prenex normal form. This is possible thanks to the
classical tautologies (for $G$ such that $x\not\in \FV(G)$),
\begin{align*}
  \forall x F \vee G &\leftrightarrow \forall x(F \vee G) &
  \exists x F \wedge G &\leftrightarrow \exists x(F \wedge G)\\
  \forall x F \wedge G &\leftrightarrow \forall x(F \wedge G) &
  \exists x F \vee G &\leftrightarrow \exists x(F \vee G)\\
  \neg\exists x F &\leftrightarrow \forall x\neg F &
  \neg\forall x F&\leftrightarrow \exists x \neg F,
\end{align*}
that allow pushing the quantifiers to the front of a formula.  In
intuitionistic logic, \emph{half} of these rules are not only
equivalences but even isomorphisms. We can also write them in a more
general form\footnote{The generalization consists in replacing the
  negation $\neg H$ (i.e. the implication $H\to\bot$) by the
  implication $H\to G$, as well as closing $G$ by a universal
  quantifier which in the case $x\not\in\FV(G)$ can be discarded.}:
\begin{align}
  \forall x F \wedge \forall x G &\cong \forall x(F \wedge G)\label{iso:quant:1}\\
  \exists x F \vee \exists x G &\cong \exists x(F \vee G)\label{iso:quant:2}\\
  \exists x F \to G &\cong \forall x (F \to G) & (\text{where } x\not\in \FV(G)).\label{iso:quant:3}
\end{align}

To see why these isomorphisms hold, it is easiest to consider a natural
deduction proof system, when formal proofs are terms of a suitable
typed lambda calculus (see for instance the intuitionistic fragment of
Table~2 from \cite{dns}) and take identity of proofs, $\equiv$, to be
the standard $=_{\beta\eta}$-relation for the lambda calculus with
product types (conjunction) and sum types (disjunction). One also has
terms for $\exists$-introduction ($\wpair{t}{p}$),
$\exists$-elimination ($\dest{p}{x}{b}{q}$), $\forall$-introduction
($\lambda x. p$) and $\forall$-elimination ($p t$), and additional rules
for $\beta$- and $\eta$-equality of terms for the quantifiers,
\begin{align*}
  (\lambda x. p) t &=_\beta p \{t/x\} \\
  \dest{\wpair{t}{p}}{x}{a}{q} &=_\beta q\{t/x\}\{p/a\}\\
  p &=_\eta \lambda x. p x\\
  p\{q/a\} &=_\eta \dest{q}{x}{b}{p\{\wpair{x}{b}/a\}},
\end{align*}
that are analogues of the $\beta$- and $\eta$-rules concerning
function types and sum types.
Given this notation, for instance, the isomorphism
\[\exists x F \to G \cong \forall x (F \to G)\] can be established
using two proof terms,
\begin{gather}
  \tag{$\phi$} \lambda a. \lambda x. \lambda b. a \wpair{x}{b}\\
  \tag{$\psi$} \lambda c. \lambda d. \dest{d}{y}{e}{c y e},
\end{gather}
by showing that $\lambda c. \phi(\psi c) =_{\beta\eta} \lambda c. c$ and $\lambda a. \psi(\phi a) =_{\beta\eta} \lambda a. a$:
\begin{multline*}
  (\lambda a. \lambda x. \lambda b. a \wpair{x}{b}) (\lambda d. \dest{d}{y}{e}{c y e}) =_{\beta}\\
  \lambda x. \lambda b. \dest{\wpair{x}{b}}{y}{e}{c y e}   =_{\beta}
  \lambda x. \lambda b. c x b   =_{\eta} c
\end{multline*}
\begin{multline*}
  (\lambda c. \lambda d. \dest{d}{y}{e}{c y e}) (\lambda x. \lambda b. a \wpair{x}{b}) =_{\beta}\\
  \lambda d. \dest{d}{y}{e}{a \wpair{y}{e}} =_{\eta}
  \lambda d. (a a_0) \{d / a_0\} =
  \lambda d. a d =_\eta a
\end{multline*}

\vfill

Similarly, we can show that a further formula isomorphism holds,
\begin{equation}
  \label{iso:quant:4}
  G\to \forall x F \cong \forall x (G \to F),
\end{equation}
when $x\notin\FV(G)$. Namely, one can take as witnessing terms the following ones:
\begin{gather}
  \tag{$\phi$} \lambda a. \lambda x. \lambda b. a b x\\
  \tag{$\psi$} \lambda c. \lambda d. \lambda x. c x d.
\end{gather}

Given the first-order formula isomorphisms (\ref{iso:quant:1}),
(\ref{iso:quant:2}), (\ref{iso:quant:3}), and (\ref{iso:quant:4}), we
shall now adopt an extended exponential polynomial notation of
formulas involving quantifiers. We write $\exists x F$ as $x F$ and
$\forall x F$ as $F^x$, the distinction between conjunctions and
existential quantifiers, and implications and universal quantifiers,
being made by a variable convention: we ``left-multiply'' and
``exponentiate'' by the lowercase $x, y, z$ in order to express
quantifiers, while if we do it with uppercase $F, G$, it means that we
are making a conjunction and implication with a generic formula.
Using this notation, the isomorphisms
(\ref{iso:quant:1})-(\ref{iso:quant:4}) acquire the form of the
following equations:
\begin{align}
  (F G)^x &= F^x G^x\tag{\ref{iso:quant:1}'}\label{iso:quant:1'}\\
  x(F+G) &= xF + xG\tag{\ref{iso:quant:2}'}\label{iso:quant:2'}\\
  G^{xF} &= (G^F)^x & (\text{where } x\notin\FV(G))\tag{\ref{iso:quant:3}'}\label{iso:quant:3'}\\
  (F^x)^G &= (F^G)^x & (\text{where } x\notin\FV(G))\tag{\ref{iso:quant:4}'}\label{iso:quant:4'}
\end{align}

This extension of HSI with rules involving the \emph{extended}
exponential polynomials thus still implies formula isomorphism. And
now we can give an interpretation of the invertible proof rules
involving the quantifiers that respect this notation:
\begin{align}
  \tag{$\forall_r$}\myinfer{\Gamma\vdash\forall x F}{\Gamma\vdash F} &&& \myinfer{(F^x)^\Gamma}{(F^\Gamma)^x} & \text{ for all } x\notin\FV(\Gamma)\ \\
  \tag{$\exists_l$}\myinfer{\exists x F, \Gamma \vdash G}{F, \Gamma \vdash G} &&& \myinfer{G^{xF \Gamma}}{\left(G^{F\Gamma}\right)^x} & \text{ for all } x\notin\FV(G,\Gamma).
\end{align}
As the invertible rules are equalities, an extension of HS from
Section~\ref{sec:highschool} for the first-order case can be defined
in the same way as before, by applying a normalization function (see
Figure~\ref{fig:norm:quantifiers}) to the premises and conclusions of
the non-invertible rules for quantifiers. Working with the first-order
extension G4i~\cite{DyckhoffNegri} of G4ip, one would have the HS
variants of the rules ($L\forall$), ($R\exists$),
($L\forall\!\!\supset$), while the invertible rule
($L\exists\!\!\supset$) can be handled using the isomorphism
(\ref{iso:quant:3}).

\begin{figure}
  \centering
    \begin{align*}
      \Base \ni b &::= p ~|~ d &
                                 \CNF \ni c &::= 1 ~|~ \left(b^{c_1}\right)^{x_1} c_2\\
      \DNF \ni d &::= c_1 \plus c_2 ~|~ c \plus d ~|~ x c & 
                                                    \ENF \ni e &::= c ~|~ d
    \end{align*}
    \begin{align*}
      \ntimes{-}{-} &: \CNF \to \CNF \to \CNF\\
      \ntimes{1}{c_2} &:= c_2\\
      \ntimes{\left(b^{c_{11}}\right)^{x}c_{12}}{c_2} &:= \left(b^{c_{11}}\right)^{x} (\ntimes{c_{12}}{c_2})\\[1em]
      \qexplogone{-}{-}{-} &: \Base \to \Vars \to \ENF \to \CNF\\
      \qexplogone{b}{((y c)^1 1)}{x} &:= \left(b^c\right)^{x,y} 1\\
      \qexplogone{b}{c}{x} &:= \left(b^c\right)^x 1\\
      \qexplogone{b}{(c_1  \plus  c_2)}{x} &:= \ntimes{((b^{c_1})^x 1)}{((b^{c_2})^x 1)}\\
      \qexplogone{b}{(c_1  \plus  d_2)}{x} &:= \ntimes{((b^{c_1})^x 1)}{(\qexplogone{b}{d_2}{x})}\\[1em]
      \explog{-}{-} &: \CNF \to \ENF \to \CNF\\
      \explog{1}{e_2} &:= 1\\
      \explog{(b^{c_{11}})^x c_{12}}{e_2} &:= \ntimes{(\qexplogone{b}{(\distribone{c_{11}}{e_2})}{x})}{(\explog{c_{12}}{e_2})}\\[1em]
      \distribex{-}{-} &: \Vars \to \ENF \to \DNF\\
      \distribex{x}{c} &:= x c\\
      \distribex{x}{(c_1\plus c_2)} &:= ((x c_1)^1)^\epsilon 1 \plus ((x c_2)^1)^\epsilon 1\\
      \distribex{x}{(c_1\plus d_2)} &:= ((x c_1)^1)^\epsilon 1 \plus \distribex{x}{d_2}\\[1em]
      \explogall{-}{-} &: \CNF \to \Vars \to \CNF\\
      \explogall{1}{x} &:= 1\\
      \explogall{((b^{c_1})^y c_2)}{x} &:= (b^{c_1})^{y,x} (\explogall{c_2}{x})
    \end{align*}
    
    \begin{align*}
      \enf{-} &: \Formula \to \ENF & \enfpos{-} &: \Formula \to \CNF\\
      \enf{p} &:= (p^1)^\epsilon 1 & \enfpos{p} &:= (p^1)^\epsilon 1\\
      \enf{F \vee G} &:= \nplus{\enf{F}}{\enf{G}} & \enfpos{F \vee G} &:= ((\nplus{\enfpos{F}}{\enfpos{G}})^1)^\epsilon 1\\
      \enf{F \wedge G} &:= \distrib{\enf{F}}{\enf{G}} & \enfpos{F \wedge G} &:= \ntimes{\enfpos{F}}{\enfpos{G}}\\
      \enf{F \to G} &:= \explog{\enfpos{G}}{\enf{F}} & \enfpos{F \to G} &:= \explog{\enfpos{G}}{\enf{F}}\\
      \enf{\exists x F} &:= \distribex{x}{\enf{F}} & \enfpos{\exists x F} &:= ((x\enfpos{F})^1)^\epsilon 1\\
      \enf{\forall x F} &:= \explogall{\enfpos{F}}{x} & \enfpos{\forall x F} &:= \explogall{\enfpos{F}}{x}
    \end{align*}

    ~
    
    $\Vars$ is the set of finite lists of first-order variables,
    denoted by $x$ or $y$; the empty list is denoted by $\epsilon$ and
    the concatenation of two lists of variables is denoted by a comma:
    $x,y$.

    \caption{Extension of the normalization functions from
      Figure~\ref{fig:norm} for the quantifiers; $x c$ denotes the
      existential quantifier $\exists x c$, while $(b^c)^x$ denotes
      the combination of the universal quantifier and implication
      $\forall x (c\to b)$; the new operations $\distribex{}{}$ and
      $\explogall{}{}$ are analogous to the operations $\distrib{}{}$
      and $\explog{}{}$ but deal with first-order variables and
      correspond to the pushing-in of quantifiers from isomorphisms
      (\ref{iso:quant:2'}) and (\ref{iso:quant:1'}.}
  \label{fig:norm:quantifiers}
\end{figure}

We have not pursued formally showing an extension of
Lemma~\ref{lem:equations}, mostly for technical reasons having to do
with formalizing syntax with binders.\footnote{Apart from the fact
  that we do not use it for formal proofs, the Coq definition of the
  functions from Figure~\ref{fig:norm:quantifiers} is simple and
  can be used to compute the formula normal form.} Hence, we do not propose to establish formally a
first-order analogue of Theorem~\ref{thm:embed} here.

What we consider as a more important consequence of the extended
exponential polynomial interpretation of the quantifiers, is the fact
that it leads to a normal form theorem for intuitionistic (minimal)
first-order formulas, which can be thought of as an analogue of the
prenex normal form for classical first-order logic, but which is
obtained by sometimes pushing the quantifiers \emph{in} (following
(\ref{iso:quant:1'}) and (\ref{iso:quant:2'})) and sometimes pushing
the quantifiers \emph{out} (following (\ref{iso:quant:3'}) and
(\ref{iso:quant:4'})), rather than always pushing them out as in the
approach for classical logic. Moreover, turning a formula into our
normal form preserves isomorphism between the original formula and its
normal form, and hence also equivalence and proof identity.

\begin{theorem}
  For every first-order formula $F$, there is an isomorphic formula
  $e\in \DNF\cup\CNF$, where the classes $\DNF$, $\CNF$, and
  $\Base$ are defined simultaneously as follows,
  \begin{align*}
    \DNF\ni d &::= c_1 \vee \cdots \vee c_n & (n\ge 2)\ \\
    \CNF\ni c &::= \forall x_1 (c_1\to b_1) \wedge \cdots \wedge \forall x_n (c_n\to b_n) & (n\ge 0)\\
    \Base \ni b &::= p ~|~ d ~|~ \exists x c,
  \end{align*}
  where $p$ denotes a prime formula, and $x$ and $x_i$ lists of
  first-order variables (potentially empty).
\end{theorem}
\begin{proof} Given $F$, the required $e$ is $\enf{F}$.  The
  normalization procedure defined in Figure~\ref{fig:norm} and
  Figure~\ref{fig:norm:quantifiers} is structurally recursive, hence
  terminating, and with range $\DNF\cup\CNF$.

  The isomorphism follows in the same way as in
  Theorem~\ref{thm:propositional:hierarchy}, taking into account the
  additional isomorphisms proven in this section for
  (\ref{iso:quant:1}), (\ref{iso:quant:2}), (\ref{iso:quant:3}), and
  (\ref{iso:quant:4}), i.e. the additional high-school identities
  (\ref{iso:quant:1'}), (\ref{iso:quant:2'}), (\ref{iso:quant:3'}),
  and (\ref{iso:quant:4'}).
\end{proof}

Since $\DNF\subset\Base$, we can also present the normal form in a way
analogous to the classical arithmetical hierarchy.

\begin{definition} The \emph{intuitionistic formula hierarchy} is
  defined by the following mutually-inductive definition of formula
  classes $\SSigma{}$ and $\PPi{}$,
  \begin{align*}
    \PPi{}\ni c &::= \forall x_1 (c_1\to b_1) \wedge \cdots \wedge \forall x_n (c_n\to b_n) & (n\ge 0)\\
    \SSigma{}\ni b &::= p ~|~ c_1 \vee \cdots \vee c_n ~|~ \exists x c & (n\ge 2),
  \end{align*}
  or, in extended exponential polynomial notation,
  \begin{align*}
    \PPi{}\ni c &::= \prod_{i=1}^{n\ge 0} \left(b_i^{c_i}\right)^{x_i} &
    \SSigma{}\ni b &::= p ~|~ \sum_{i=1}^{n\ge 2} c_i ~|~ x c,
  \end{align*}
  where $p$ denotes a prime formula, and $x$ and $x_i$ lists of
  first-order variables (potentially empty).
\end{definition}

The name ``hierarchy'' is justified by the fact that the classes
$\SSigma{}$ and $\PPi{}$ are interleaved, being defined by a truly
mutual inductive definition. A lower level of a class is embedded in
the immediately succeeding higher level by the formula transformations
$b\mapsto (b^1)^{\epsilon}$ and $c\mapsto \epsilon c$, which are
clearly isomorphisms (recall that $\epsilon$ denotes the empty list of
first-order variables). It should also be clear that a higher level
cannot be embedded in a lower one, but, in order to make these
statements more precise, we define a particular linearization of the
hierarchy.

\begin{definition}
  The \emph{intuitionistic arithmetical hierarchy} is defined from the
  intuitionistic formula hierarchy by assigning levels,
  $\SSigma{n}, \PPi{n}$, for $n\in\mathbb{N}$, to the formula classes
  $\SSigma{}$ and $\PPi{}$, in the following way:
  \begin{align*}
    \PPi{0}\ni c &::= p^11 & p\text{ is a prime formula }\\
    \SSigma{0}\ni b &::= p & p\text{ is a prime formula }\\
    \PPi{n+1}\ni c &::= \prod_{i=1}^{m\ge 0} \left(b_i^{c_i}\right)^{x_i} & n= \max_{i=1}^m\{k ~|~ b_i\in\SSigma{k}\}\\
    \SSigma{n+1}\ni b &::= \sum_{i=1}^{m\ge 2} c_i ~|~ x c & n= \max_{i=1}^m\{k ~|~ c_i\in\PPi{k}\} \text{ or } c\in\PPi{n}.
  \end{align*}
  We also extend the relation ``$\in$'' from formulas satisfying the
  inductive definition to all formulas, in the following way:
  \begin{itemize}
  \item $F\in\PPi{n}$ iff $\enf{F}\in\PPi{n}$,
  \item $F\in\SSigma{n+1}$ iff $\enf{F}\in\SSigma{n+1}$.
  \end{itemize}
\end{definition}
\begin{remark}\label{rem:cibi}
  Note, that, when determining the level $n$ for the class $\PPi{n}$,
  we do not take into account the level of the sub-formulas $c_i$ from
  $\forall x_1 (c_1\to b_1) \wedge \cdots \wedge \forall x_n (c_n\to
  b_n)$, but only the level of the sub-formulas $b_i$.
\end{remark}
\begin{remark}
  Although, due to the syntactic nature of the previous definition,
  the direct inclusions $\SSigma{n}\subseteq\SSigma{n+1}$ and
  $\PPi{n}\subseteq\PPi{n+1}$ do not hold, we do have the following:
  \begin{align*}
    b\in\SSigma{n} &\text{ implies that there exists } b'\in\SSigma{n+1} \text{ such that } b\cong b'; \text{ and }\\
    c\in\PPi{n} &\text{ implies that there exists } c'\in\PPi{n+1} \text{ such that } c\cong c'.
  \end{align*}
  This is proved by induction on $n$, simultaneously for the two
  statements, as follows.

  Base case. If $b\in\SSigma{0}$, then $b$ is of form $p$, hence
  $b' := \epsilon (p^1 1) \cong b$ and $b'\in \SSigma{1}$ since
  $p^1 1\in\PPi{0}$. If $c\in\PPi{0}$, then $c$ is of form $p^11$,
  hence $c':=c\in\PPi{1}$.

  Induction case. If $b\in\SSigma{n+1}$, then $b$ is either of form
  $c_1+\cdots+c_m$ or $x c$, with
  $n= \max_{i=1}^m\{k ~|~ c_i\in\PPi{k}\}$ or $c\in\PPi{n}$. By the
  induction hypothesis, we find one $c_i'\in\PPi{n+1}$ such that
  $c_i\cong c_i'$, or $c\in\PPi{n+1}$ such that $c\cong c'$. Then,
  $c_1+\cdots+c_{i-1}+c_i'+c_{i+1}+\cdots+c_m$ or $x c'$ belongs to
  $\SSigma{n+2}$, and
  $c\cong c_1+\cdots+c_{i-1}+c_i'+c_{i+1}+\cdots+c_m$ or
  $c \cong x c'$. If $c\in\PPi{n+1}$, then $c$ is of form
  $\prod_{i=1}^{m\ge 0} \left(b_i^{c_i}\right)^{x_i}$ with
  $n= \max_{i=1}^m\{k ~|~ b_i\in\SSigma{k}\}$. By the induction
  hypothesis, we find one $b_i'\in\SSigma{n+1}$ such that
  $b_i\cong b_i'$. Then,
  $c\cong\left(b_1^{c_1}\right)^{x_1}\cdots\left(b_{i-1}^{c_{i-1}}\right)^{x_{i-1}}\left((b_i')^{c_i}\right)^{x_i}\left(b_{i+1}^{c_{i+1}}\right)^{x_{i+1}}\cdots\left(b_m^{c_m}\right)^{x_m}\in\PPi{n+2}$.
  
\end{remark}


We shall now show a connection between the intuitionistic arithmetical
hierarchy and the classical arithmetical hierarchy. From this
connection, it will follow that the intuitionistic hierarchy is
proper, that is, that the interleaving of the classes $\SSigma{}$ and
$\PPi{}$ is such that each class is properly extended by the next one
-- the hierarchy does not collapse.

\subsection{Properness of the intuitionistic hierarchy via a relation
  to the classical one}

Recall the definition of the classes $\Sigma^0_n,\Pi^0_n$ of the
classical arithmetical hierarchy\footnote{To distinguish easily
  between the classical and the intuitionistic hierarchy, the levels
  of the classical hierarchy are denoted in regular face and with a
  $0$-superscript, while the intuitionistic ones are in bold face and
  have no superscript.} (see for instance Section~2.3 of
\cite{SchwichtenbergW2011}):
\begin{itemize}
\item $F\in\Pi^0_{n+1}$ iff $F$ is classically equivalent to a formula
  of the form $\forall x G$ where $G\in\Sigma^0_n$,
\item $F\in\Sigma^0_{n+1}$ iff $F$ is classically equivalent to a
  formula of the form $\exists x G$ where $G\in\Pi^0_n$,
\item $F\in\Sigma^0_0$ and $F\in\Pi^0_0$ iff $F$ is classically
  equivalent to an elementary relation $E$.
\end{itemize}
Note that, in the classical case, a basic theory of arithmetic is
assumed on top of classical first-order logic, a theory able to code
syntax and define Kleene's $T$ predicate, in order to show that the
hierarchy classifies all formulas and that its levels are properly
increasing. In the intuitionistic case, we did not need to assume a
theory, working with pure intuitionistic first-order logic. We will
assume now, however, that every elementary relation $E$ is given by an
atomic predicate with the same name $E$.

Say that a formula $F$ is \emph{classically represented in
  $\SSigma{n}$ (or $\PPi{n}$)} when there is a formula
$F'\in\SSigma{n}$ (or $\PPi{n}$) such that $F$ and $F'$ are
\emph{classically} equivalent. This means that $F$ itself is not
necessarily in $\SSigma{n}$ or $\PPi{n}$ (because $\enf{F}$ is not
there for the specific $n$), but that $F$ is classically equivalent to
a formula $F'$ such that $\enf{F'}$ is in $\SSigma{n}$ or
$\PPi{n}$. From this discussion, it is clear that we do not get that
$\Sigma^0_n\subseteq \SSigma{n}$ or
$\Pi^0_n\subseteq\PPi{n}$. However, we do get the following
related proposition.
\begin{proposition}\label{prop:classrep}
  Every formula of the classical arithmetical hierarchy is classically
  represented at the corresponding level in the intuitionistic
  arithmetical hierarchy. That is: if $F\in\Sigma^0_n$, then $F$ is
  classically represented in $\SSigma{n}$; if $F\in\Pi^0_n$, then $F$
  is classically represented in $\PPi{n}$.
\end{proposition}
\begin{proof}
  If $F\in \Sigma^0_0$ (i.e., $F\in \Pi^0_0$), then $F$ is classically
  equivalent to a predicate $E$ (and $E^1 1$), and, since $E$ is
  prime, $E\in\SSigma{0}$ (and $E^1 1\in\PPi{0}$) and hence $F$ is
  classically represented in $\SSigma{0}$ and $\PPi{0}$.

  If $F\in \Pi^0_{n+1}$, then $F$ is classically equivalent to a
  formula of form $\forall x G$, where $G\in\Sigma^0_n$. Assuming that
  $G$ is classically represented by $G'$ in $\SSigma{n}$ (induction
  hypothesis), then $F$ is classically represented in $\PPi{n+1}$,
  because $\forall x G$ is classically equivalent to
  $\forall x (\top\to \enf{G'})$ which is in $\PPi{n+1}$.

  If $F\in \Sigma^0_{n+1}$, then $F$ is classically equivalent to a
  formula of form $\exists x G$, where $G\in\Pi^0_n$. Assuming that
  $G$ is classically represented by $G'$ in $\PPi{n}$ (induction
  hypothesis), then $F$ is classically represented in $\SSigma{n+1}$,
  because $\exists x G$ is classically equivalent to
  $\exists x \enf{G'}$ which is in $\SSigma{n+1}$.
\end{proof}
\begin{remark}\label{rem:prenex}
  As can be seen from the proof of the previous proposition, a formula
  of the classical arithmetical hierarchy is actually classically
  represented in the intuitionistic hierarchy essentially by its
  representation in the \emph{classical} hierarchy. For example, if
  $F\in \Pi^0_4$, then $F$ is classically equivalent to a formula of
  form $\forall x \exists y \forall z \exists u E$, which is in
  $\PPi{4}$ because
  $\enf{\forall x \exists y \forall z \exists u E} = \forall x
  (\top\to\exists y \forall z (\top\to\exists u E))$.  In other words,
  formulas in prenex normal form and with alternating quantifiers are
  represented by (essentially) themselves in both the classical and
  the intuitionistic hierarchy.
\end{remark}

As for the plain inclusion of the intuitionistic arithmetical
hierarchy in the classical one, $\SSigma{n}\subseteq\Sigma^0_{n}$ and
$\PPi{n}\subseteq\Pi^0_{n}$, it does not follow either: although
isomorphism implies intuitionistic equivalence and hence classical
equivalence, because of Remark~\ref{rem:cibi}, we cannot prove the
inclusion $\PPi{n}\subseteq\Pi^0_{n}$ (and hence
$\SSigma{n}\subseteq\Sigma^0_{n}$) in general. However, for the class
of formulas in the prenex normal form
$\QQ_1 x_1\QQ_2 x_2\cdots\QQ_n x_n P$, for $\QQ_i\in\{\forall,\exists\}$,
$\QQ_{i}\neq\QQ_{i+1}$, and $P$ a prime formula, the inclusion does
hold.

\begin{proposition}\label{prop:prenex}
  Let $F$ be in the prenex normal form
  $\QQ_n x_n\QQ_{n-1}x_{n-1}\cdots\QQ_1 x_1 P$ with alternating quantifiers
  (i.e., $\QQ_i\neq\QQ_{i+1}$ for all $i$). Then: if $F\in\SSigma{n}$,
  then $F\in\Sigma^0_{n}$; if $F\in\PPi{n}$, then $F\in\Pi^0_{n}$.
\end{proposition}
\begin{proof}
  We first prove that, given a formula $G$ such that
  $\enf{G}\in\PPi{}$ (i.e., in $\CNF$), we have that
  $\enf{G}=\enfpos{G}$.

  The proof is by induction on the structure of $G$. Because
  $\enf{G}\in\CNF$, the cases where $G$ is a disjunction or an
  existential quantifier are not possible. For the base case $G = p$,
  and the induction cases $G = G_1\to G_2$ and $G = \forall x G_1$,
  the definitions of $\enf{}$ and $\enfpos{}$ are identical, so there
  is even no need to invoke the induction hypothesis. There is only
  one induction case left to treat.
  \begin{description}
  \item[$G = G_1\wedge G_2$] Note that both $\enf{G_1}$ and
    $\enf{G_2}$ must be in $\CNF$, because otherwise, by the
    definition of $\distrib{}{}$, $\enf{G}$ would not have been in
    $\CNF$. We have
    $\distrib{\enf{G_1}}{\enf{G_2}} =
    \ntimes{\enfpos{G_1}}{\enfpos{G_2}}$ by the induction hypothesis 
    and because $\distrib{}{} = \ntimes{}{}$ when both arguments of
    $\distrib{}{}$ belong to $\CNF$.
  \end{description}

  Now, we can prove the proposition, by well-founded induction on
  $n\ge 0$:
  \begin{description}
  \item[base case]   In the base case, $F$ is of form $P$, and we have
    $P\in\SSigma{0}\cap\Sigma^0_0\cap\PPi{0}\cap\Pi^0_0$.
  \item[induction case] Suppose that the proposition holds for $n$ and
    that we want to show it for $n+1$. We consider the two possible
    cases for the quantifier $\QQ_{n+1}$.
    \begin{enumerate}
    \item Suppose $\exists x_{n+1} G\in\SSigma{n+1}$, i.e.,
      $\enf{\exists x_{n+1} G}\in\SSigma{n+1}$. We have
      $\enf{\exists x_{n+1} G}=\distribex{x_{n+1}}{\enf{G}}=\exists
      x_{n+1}\enf{G}$, by the definition of $\distribex{}{}$, since
      $\enf{G}\in\CNF$ because $G$ is either $P$ or starts with a
      universal quantifier. From
      $\exists x_{n+1}\enf{G}\in\SSigma{n+1}$, we have that
      $\enf{G}\in\PPi{n}$, so $G\in\PPi{n}$ and, by induction
      hypothesis, $G\in \Pi^0_n$. Therefore
      $\exists x_{n+1} G\in\Sigma^0_{n+1}$.
    \item Suppose $\forall x_{n+1} G\in\PPi{n+1}$, i.e.,
      $\enf{\forall x_{n+1} G}\in\PPi{n+1}$. $G$ is either $P$ or
      starts with an existential quantifier.  If $G=P$, then
      $\enf{\forall x_{n+1} G} = \explogall{\enfpos{G}}{x_{n+1}} =
      \forall x_{n+1}(\top\to P) \in \PPi{n+1}$ implies
      $P\in\SSigma{n}$ (in this case, $n=0$). If
      $G = \exists x_n H$, then
      $\enf{\forall x_{n+1} G} = \explogall{\enfpos{G}}{x_{n+1}} =
      \forall x_{n+1}(\top\to \exists x_n \enfpos{H}) = \forall
      x_{n+1}(\top\to \exists x_n \enf{H}) \in \PPi{n+1}$ implies
      $\enf{H}\in\PPi{n-1}$ i.e., $H\in\PPi{n-1}$; here we were able
      to use the equation $\enf{H}=\enfpos{H}$, because
      $\enf{H}\in\CNF$ since $H$ is either $P$ or starts with a
      universal quantifier.

      We have got either $P\in\SSigma{n}$ or $H\in\PPi{n-1}$ and, by
      using the induction hypothesis, we get either
      $\forall x_{n+1}P \in\Pi^0_{n+1}$ or
      $\forall x_{n+1}\exists x_{n} H\in\Pi^0_{n+1}$.
    \end{enumerate}
  \end{description}
\end{proof}

We can now prove that the intuitionistic arithmetical hierarchy, and
hence the more general intuitionistic formula hierarchy, is proper.

\begin{corollary} For $n\ge 0$,
  $\SSigma{n}\subsetneq\SSigma{n+1}$,
  $\SSigma{n}\subsetneq\PPi{n+1}$,
  $\PPi{n}\subsetneq\SSigma{n+1}$, and
  $\PPi{n}\subsetneq\PPi{n+1}$.
\end{corollary}
\begin{proof}
  We will prove only one of the statements,
  $\SSigma{n}\subsetneq\PPi{n+1}$, as the proofs of the other three
  are analogous.

  Let $F\in \Pi^0_{n+1}\setminus\Sigma^0_n$. Such a formula exists by
  the properness of the classical arithmetical hierarchy.\footnote{The
    proofs are constructive, that is, there are explicit such formulas
    -- again, one can see Section~2.3 of \cite{SchwichtenbergW2011}
    for concrete examples.}  By Proposition~\ref{prop:classrep}, $F$
  is classically represented in $\PPi{n+1}$ by a formula
  $F'\in\PPi{n+1}$. This $F'$ will be our example of a formula in
  $\PPi{n+1}\setminus\SSigma{n}$.

  We also know that $F\not\in\Sigma^0_n$ and hence
  $F'\not\in\Sigma^0_n$. Since by the proof of
  Proposition~\ref{prop:classrep} $F'$ is in prenex form and with
  alternating quantifiers (see Remark~\ref{rem:prenex}), from
  Proposition~\ref{prop:prenex}, we have that $F'\not\in\SSigma{n}$
  (we used the (constructive) contraposition of ``$F'\in\SSigma{n}$
  implies $F'\in\Sigma^0_n$'').
\end{proof}

\subsection{Examples of using the intuitionistic hierarchies}

In this last subsection, we give two examples that show our
hierarchies at work. In the first example, we categorize the formulas
of the well known Nishimura lattice.  In the second one, we show how
the hierarchy allows one to simplify definitions that are stated by
induction on the structure of formulas, concretely the double-negation
translation. The first example is an application of the intuitionistic
arithmetical hierarchy (it's propositional fragment), while the second
is an application of the \emph{formula} hierarchy directly.

\begin{example}[Nishimura lattice]
  The \emph{basic formulas of the Nishimura
    lattice}~\cite{Nishimura1960,sep-logic-intuitionistic} is the
  following sequence of formulas:
  \begin{align}
    F_{\infty}&:= 1 & F_0&:=\bot & F_1&:=p & F_2&:=\bot^p & F_{2n+3} &:= F_{2n+1}+F_{2n+2} & F_{2n+4} &:= (F_{2n+1})^{F_{2n+3}},\label{nishimura:formulas}
  \end{align}
  for $n\ge 0$.

  This sequence characterizes all propositional formulas of
  intuitionistic logic \emph{that use at most one prime formula, $p$}, in
  the following sense:
  \begin{itemize}
  \item any formula is equivalent to one of the basic formulas;
  \item no two of the basic formulas are intuitionistically equivalent to each other;
  \item only the following intuitionistic implications between basic formulas hold
    (and the intuitionistic implications that follow from them):
    \begin{align*}
      F_{2n} &\to F_{2n+1}&
                            F_{2n+1} &\to F_{2n+3}&
                                                    F_{2n+1} &\to F_{2n+4}&
                                                                            F_0&\to F_2.
    \end{align*}
  \end{itemize}
  So, for the case of propositional intuitionistic logic restricted to
  one prime formula $p$, the basic formulas of the Nishimura lattice
  form a kind of a hierarchy. In the following diagram, we classify
  them in our intuitionistic arithmetical hierarchy, while at the same
  time marking the implications that hold between the basic formulas.
  
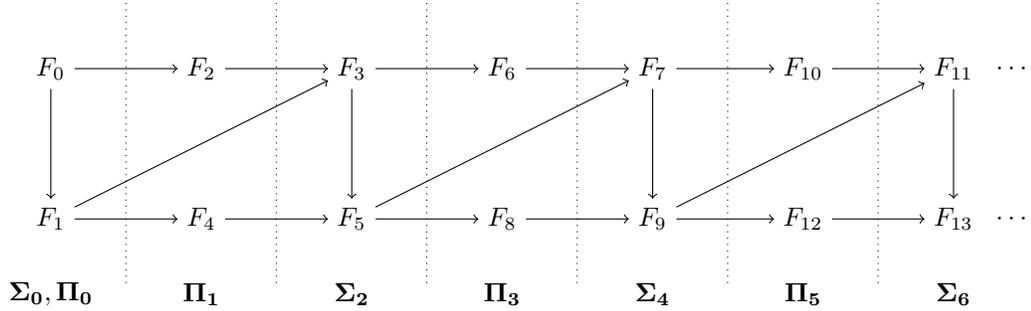
\begin{center}
    \begin{tikzpicture}[node distance=2cm]
      \node(1a){$F_0$};
      \node(1b)[below of=1a]{$F_1$};
      \path[->][->] (1a) edge (1b);
      \node(2a)[right of=1a]{$F_2$};
      \node(2b)[right of=1b]{$F_4$};
      \path[->] (1a) edge (2a);
      \path[->] (1b) edge (2b);
      \node(3a)[right of=2a]{$F_3$};
      \node(3b)[right of=2b]{$F_5$};
      \path[->] (2a) edge (3a);
      \path[->] (2b) edge (3b);
      \path[->] (1b) edge (3a);
      \path[->] (3a) edge (3b);
      \node(4a)[right of=3a]{$F_6$};
      \node(4b)[right of=3b]{$F_8$};
      \path[->] (3a) edge (4a);
      \path[->] (3b) edge (4b);
      \node(5a)[right of=4a]{$F_7$};
      \node(5b)[right of=4b]{$F_9$};
      \path[->] (4a) edge (5a);
      \path[->] (4b) edge (5b);
      \path[->] (3b) edge (5a);
      \path[->] (5a) edge (5b);
      \node(6a)[right of=5a]{$F_{10}$};
      \node(6b)[right of=5b]{$F_{12}$};
      \path[->] (5a) edge (6a);
      \path[->] (5b) edge (6b);
      \node(7a)[right of=6a]{$F_{11}$};
      \node(7b)[right of=6b]{$F_{13}$};
      \path[->] (6a) edge (7a);
      \path[->] (6b) edge (7b);
      \path[->] (5b) edge (7a);
      \path[->] (7a) edge (7b);
      \node(8a)[right of=7a, node distance=0.8cm]{$\cdots$};
      \node(8b)[right of=7b, node distance=0.8cm]{$\cdots$};

      \node(1c)[below of=1b, node distance=1.0cm]{$\SSigma{0},\PPi{0}$};
      \node(2c)[below of=2b, node distance=1.0cm]{$\PPi{1}$};
      \node(3c)[below of=3b, node distance=1.0cm]{$\SSigma{2}$};
      \node(4c)[below of=4b, node distance=1.0cm]{$\PPi{3}$};
      \node(5c)[below of=5b, node distance=1.0cm]{$\SSigma{4}$};
      \node(6c)[below of=6b, node distance=1.0cm]{$\PPi{5}$};
      \node(7c)[below of=7b, node distance=1.0cm]{$\SSigma{6}$};

      \node(12a)[above right of=1a, node distance=1.41cm]{};
      \node(12b)[below right of=1b, node distance=1.41cm]{};
      \path[dotted] (12a) edge (12b);
      \node(22a)[above right of=2a, node distance=1.41cm]{};
      \node(22b)[below right of=2b, node distance=1.41cm]{};
      \path[dotted] (22a) edge (22b);
      \node(32a)[above right of=3a, node distance=1.41cm]{};
      \node(32b)[below right of=3b, node distance=1.41cm]{};
      \path[dotted] (32a) edge (32b);
      \node(42a)[above right of=4a, node distance=1.41cm]{};
      \node(42b)[below right of=4b, node distance=1.41cm]{};
      \path[dotted] (42a) edge (42b);
      \node(52a)[above right of=5a, node distance=1.41cm]{};
      \node(52b)[below right of=5b, node distance=1.41cm]{};
      \path[dotted] (52a) edge (52b);
      \node(62a)[above right of=6a, node distance=1.41cm]{};
      \node(62b)[below right of=6b, node distance=1.41cm]{};
      \path[dotted] (62a) edge (62b);
    \end{tikzpicture}    
  \end{center}
  Calculating the level of a formula is an easy application of the
  normalization function $\enf{\cdot}$. Only the propositional part of
  the hierarchy is needed for this example, and therefore the
  normalization procedure from Figure~\ref{fig:norm} suffices. The
  symbol $\bot$ is treated just as any other prime formula by the
  normalization function. As a general rule, we get that:
  \begin{align}\label{nishimura:levels}
    \PPi{n} &\ni F_{2n}, F_{2n+2} & \text{ for odd }n&\ge 1;&
    \SSigma{n} &\ni F_{2n-1}, F_{2n+1} & \text{ for even }n&\ge 2.
  \end{align}
  The characterization of the basic formulas regarding
  (non)-equivalence that was proven by Nishimura provides us with
  examples of:
  \begin{itemize}
  \item pairs of non-equivalent and hence non-isomorphic formulas at
    each of the levels $\SSigma{n}$ ($n$-even) and $\PPi{n}$ ($n$-odd)
    of the intuitionistic arithmetical hierarchy -- showing that each
    of those levels contains at least two formulas (exactly two
    formulas up to intuitionistic equivalence);
  \item examples of formulas for each of
    $\PPi{n+1}\setminus\SSigma{n}$ ($n$-even) and
    $\SSigma{n+1}\setminus\PPi{n}$ ($n$-odd) -- showing that each of
    those levels of the hierarchy is a proper extension of the previous
    one: an example formula cannot be in both levels because, when a
    formula is in a higher level, since it is not intuitionistically
    equivalent to any other basic formula, it cannot be isomorphic to
    any other basic formula, and in particular it cannot be isomorphic
    to a formula at a lower level.
  \end{itemize}
  Note that, although no special interpretation of $\bot$ needs to be
  assumed in order to \emph{classify} the Nishimura formulas in the
  intuitionistic hierarchy, we do have to make the same assumptions as
  \cite{Nishimura1960} for the examples of non-equivalent formulas
  from the previous two bullet points: the assumption of usual
  intuitionistic interpretation of $\bot$ and the assumption of
  formulas being constructed from at most one atom $p$ (besides the
  atom $\bot$).  
\end{example}

\begin{example}[Double-negation translation]
  Consider the double-negation translation $F^\bot$ of formula $F$,
  \begin{align*}
    F^\bot &:= (F_\bot \to \bot)\to \bot\\
    P_\bot &:= P & P\text{-prime}\\
    (F\wedge G)_\bot &:= F_\bot \wedge G_\bot\\
    (F\vee G)_\bot &:= F_\bot \vee G_\bot\\
    (F\to G)_\bot &:= F_\bot \to G^\bot\\
    (\exists x F)_\bot &:= \exists x F_\bot\\
    (\forall x F)_\bot &:= \forall x F^\bot.
  \end{align*}
  This translation is known as the call-by-value translation, a
  version of the Kuroda translation that also works in minimal logic
  (i.e., the $\bot$-elimination rule is not needed to establish the
  soundness of the translation).

  Now consider the following translation, derived from the above one
  by considering only formulas in normal form,
  $e\in\SSigma{}\cup\PPi{}$ -- since any formula can be brought to
  normal form, the translation is applicable to all formulas.
  \begin{align*}
    e^\bot &:= (e_\bot \to \bot)\to \bot\\
    p_\bot &:= p & p\text{-prime}\\
    (\forall x_1 (c_1\to b_1) \wedge \cdots \wedge \forall x_n (c_n\to b_n))_\bot &:= \forall x_1 ({c_1}_\bot\to {b_1}^\bot) \wedge \cdots \wedge \forall x_n ({c_n}_\bot\to {b_n}^\bot)\\
    (c_1 \vee \cdots \vee c_n)_\bot &:= {c_1}_\bot \vee \cdots \vee {c_n}_\bot\\
    (\exists x c)_\bot &:= \exists x c_\bot.
  \end{align*}
  The advantage of using normalized formulas instead of general
  formulas in this definition of double-negation translation is
  double:
  \begin{enumerate}
  \item Instead of considering six general cases when defining
    $(\cdot)_\bot$ in the original double-negation translation, one
    case for each logical connective, we only need to consider four
    specialized cases in our version of the translation; although, one
    does now have to work with a synthetic and ``vectorized'' logical
    connective for the case covering $\forall$ and $\to$;
  \item In our translation, we can get away with not adding a double
    negation after the universal quantifier: this works because
    formulas in normal form are guaranteed to have an implication
    below the universal quantifier and the equivalence between
    $\forall x \neg\neg(F\to\neg\neg G)$ and
    $\forall x (F\to\neg\neg G)$ is provable in intuitionistic
    (minimal) logic.
  \end{enumerate}
\end{example}


\section{Conclusion}
\label{sec:conclusion}

In this paper, we employ the exponential polynomial aspect of formulas
to study the structure of proofs and formula equivalence, allowing a
fresh perspective on intuitionistic proof theory and a first link
to other areas of computer science and mathematics, which one could
potentially exploit to obtain new results in logic.

Indeed, as we saw, seeing proof rules as relations (inequalities)
between exponential polynomials allowed us to define HS. As far as we
know, this is the first proof formalism for intuitionistic logic that
dispenses with invertible proof rules. We thus believe it to be a
contribution to the study of identity of proofs \cite{Dosen2003}, an
open problem identified already by Kreisel~\cite{Kreisel1971} and
Prawitz~\cite{Prawitz1971}.  Investigating whether HS-notation alone
is enough to define identity of proofs is a topic of future work. In a
related published work~\cite{explog}, we have studied the equational
theory of $=_{\beta\eta}$ for the lambda calculus with sum types
(i.e. intuitionistic natural deduction), after terms are coerced to a
type normal form very similar to the one shown in
Figure~\ref{fig:norm}. A new decomposition of the so far standard
identity of proofs relation of $\beta\eta$-equality is proposed there,
from which one can also see that the permutations of invertible rules
are an obstacle for comparing derivations, and that a natural
deduction calculus is less suitable than a sequent calculus for
studying proof identities.

The idea that invertible proof rules of sequent calculus should be
treated in blocks, inside which the order of application of rules does not
matter, is present in the approach to \emph{focused sequent calculi} such as
Liang and Miller's intuitionistic system LJF~\cite{LiangMiller2007},
inspired by previous work on linear logic by Andreoli
\cite{andreoli}. The difference between our approach and LJF is that
working on the top-most connectives of a sequent (formula) does not
allow one to apply all applicable type isomorphisms as sequent
transformations, and, as a consequence, a focusing proof proceeds in an
alternation of invertible and non-invertible blocks of proof rules --
the invertible rules still being present.

As we have shown, all the rules of the HS calculus are interpreted as strict
inequalities of natural numbers when atoms are instantiated with appropriately
chosen values. Because our exponential polynomials are manifestly monotonic when
interpreted as functions (i.e. if $f(x)$ is an exponential polynomial containing
the variable $x$, and $n\le m$, then $f(n)\le f(m)$), one might consider
allowing the application of the inference rules in Figure~\ref{fig:hs} not only
at the top level of the sequent, but also \emph{deeply} inside the formulas
themselves. This could lead to a \emph{deep inference}~\cite{Guglielmi07} calculus
in the style of G4ip. Moreover, because of the aforementioned monotonicity, it
should be possible to extend the results of Section~\ref{sec:inequality} to
ensure that this calculus is terminating as well. Note that intuitionistic
calculi presented as deep inference systems (see e.g.
\cite{GuenotStrassburger14}) usually have an explicit contraction rule, which
precludes such a termination argument.

Compared to more traditional sequent calculi, HS is maybe closer to
Vorob'ev's original calculus \cite{Vorobev1958}, that contains
distributivity proof rules (i.e.~(\ref{hsi:distr})), than Dyckhoff's
\cite{Dyckhoff1992} and Hudelmaier's \cite{Hudelmaier}, which do not
apply such proof rules.

Furthermore, we showed that the inequality interpretation allows us to
formulate a simple termination argument for proof search in
intuitionistic propositional logic. This could be potentially useful
for automated and inductive theorem proving, but it is not the subject
of this paper to go into details of how to actually perform proof
search.

The formula hierarchy from Section~\ref{sec:hierarchy}
appears to be the first systematic classification of first-order
formulas preserving isomorphism. One could also argue that it is among
the simplest hierarchies for intuitionistic logic so far and analogous
to the classical arithmetical hierarchy. Our hierarchy could also be
used as an alternative one in the context of classical logic, however,
the benefits of preserving isomorphism of formulas in the classical
setting are not clear, as, at this moment, it is not clear whether
there is a meaningful (non-degenerated and non-trivial) notion of proof identity for
classical proof systems (see sections~5 and~6 of \cite{Dosen2003}).

Previous intuitionistic hierarchies that we know of are the ones of
Mints~\cite{SchubertUZ2015,Mints1968}, Leivant~\cite{Leivant1981},
Burr~\cite{Burr2000}, and Fleischmann~\cite{Fleischmann2010}. Mints'
hierarchy of formulas is restricted to ones not containing existential
quantification (in a positive position); given this restriction, every
formula in a level $\Sigma_n$ or $\Pi_n$ of Mints' hierarchy is
classically equivalent to a formula in the classical arithmetical
hierarchy at the corresponding level $\Sigma^0_n$ or $\Pi^0_n$. With
the aim of showing complexity bounds on termination of proof search,
the line of work on Mints' hierarchy has recently been continued by
Schubert, Urzyczyn, and Zdanowski~\cite{SchubertUZ2015}. Leivant
defined formula classes for intuitionistic logic based on
implicational complexity, that is the depth of negative nestings of
implications. Burr proposes a formula class $\Phi_n$, that over
classical logic coincides with the class $\Pi^0_1$ of the arithmetical
hierarchy, however he gives no ``reasonable counterpart'' for the
classes $\Sigma^0_n$ when $n\ge 2$. Fleischmann introduces inductive
operators for universal, $\mathcal{U}(\cdot,\cdot)$, and existential,
$\mathcal{E}(\cdot)$, closure of sets of formulas, showing they can be
used to obtain a number of different hierarchies, one of them
coinciding with Burr's hierarchy, and then uses these operators to
obtain model theoretic preservation theorems. It is not clear how to
obtain our hierarchy using Fleischmann's operators, in the form in
which they are given; also, our hierarchy classifies formulas while
preserving their isomorphism, not only their equivalence.

It might also be interesting to notice a connection with the class of
coherent or geometric formulas~\cite{BezemCoquand}: using our
notation, they can be written in the form
$(x_1 c_1+\cdots+x_n c_n)^{p_1\cdots p_m}\in\CNF$, where $c_i\in\CNF$
do not contain implications (except for trivial implications of
the form $b^1$).

Finally, as we saw in Section~\ref{sec:hierarchy}, the formula
hierarchy gives a normal form for first-order formulas, which can be
used to simplify definitions or arguments that proceed by induction on
the structure of formulas. One could perhaps also see this normal form
as an intuitionistic analogue of the prenex normal form.


\subsection*{Acknowledgement}
This work has been funded by ERC Advanced Grant ProofCert and has
benefited from discussions with our colleagues Zakaria Chihani, Anupam
Das, Nicolas Guenot, and Dale Miller. The remarks of the anonymous
reviewers significantly improved the paper.


\bibliographystyle{unsrt}
\bibliography{high-school}


\end{document}